\documentclass{amsart}[12pt]

\usepackage{amsfonts,amssymb,amsmath,amscd,amsthm, mathrsfs}
\usepackage{graphicx}
\usepackage{epstopdf}
\usepackage{hyperref}
\usepackage{upgreek}

\newtheorem{thm}{Theorem}[section]
\newtheorem{prop}{Proposition}[section]
\newtheorem{lm}{Lemma}[section]
\newtheorem{cor}{Corollary}[section]
\newtheorem{conj}{Conjecture}[section]

\theoremstyle{definition}
\newtheorem{dfn}{Definition}[section]

\theoremstyle{remark}
\newtheorem{rem}{Remark}[section]

\newcommand{\mf}[1]{\mathfrak{#1}}
\newcommand{\mc}[1]{\mathcal{#1}}
\newcommand{\rr}{\mathbb{R}}
\newcommand{\zz}{\mathbb{Z}}
\newcommand{\cc}{\mathbb{C}}
\newcommand{\J}{\mathbb{J}}

\newcommand{\N}{\mathbf{N}}
\newcommand{\I}{\mathbf{I}}

\newcommand{\e}{\mathbf{e}}

\newcommand{\fn}{\mf{N}}
\newcommand{\cn}{\mathcal{N}}

\newcommand{\ey}{\frac{1}{2}}

\newcommand{\R}{\mf{R}}
\newcommand{\A}{\mc{A}}
\newcommand{\ao}{\mc{A}_{\om}}

\newcommand{\dl}{\delta}

\newcommand{\ep}{\varepsilon}

\newcommand{\Lmd}{\Lambda}
\newcommand{\lmd}{\lambda}
\newcommand{\om}{\omega}

\newcommand{\sg}{\sigma}
\newcommand{\zt}{\zeta}

\newcommand{\tht}{\theta}

\newcommand{\ldn}{\Lambda^{D_N}}
\newcommand{\hld}{\hat{\Lambda}^{D_N}}
\newcommand{\lzn}{ \Lmd^{\zz_{N}}}

\newcommand{\qd}{\dot{q}}
\newcommand{\xd}{\dot{x}}

\newcommand{\zd}{\dot{z}}

\newcommand{\qt}{\tilde{q}}
\newcommand{\yt}{\tilde{y}}
\newcommand{\xt}{\tilde{x}}
\newcommand{\ztl}{\tilde{z}}

\newcommand{\qo}{q^{\om}}
\newcommand{\qe}{q^{\ep}}
\newcommand{\qn}{q^n}

\newcommand{\qf}{\mf{q}}
\newcommand{\xf}{\mf{x}}
\newcommand{\yf}{\mf{y}}
\newcommand{\zf}{\mf{z}}

\newcommand{\etn}{\eta^n}
\newcommand{\ztn}{\zt^n}

\newcommand{\ej}{e^{\J \om t}}

\newcommand{\hnm}{H^1_{2\pi}}
\newcommand{\lnm}{L^2_{2\pi}}

\begin{document}

	\title[connecting linear chains]{Connecting planar linear chains in the spatial $N$-body problem}
	\author{Guowei Yu}
	\email{guowei.yu@unito.it}

	\thanks{The author acknowledges the support of the ERC Advanced Grant 2013 No.  339958 ``Complex Patterns for Strongly Interacting Dynamical Systems - COMPAT''}
	%\date{25/10/2017}
	
	\address{Dipartimento di Matematica ``Giuseppe Peano'', Universit\`a degli Studi di Torino, Italy}

	\begin{abstract} 
	     The family of planar linear chains are found as collision-free action minimizers of the spatial $N$-body problem with equal masses under $D_N$ or $D_N \times \zz_2$-symmetry constraint and different types of topological constraints. This generalizes a previous result by the author in \cite{Y15c} for the planar $N$-body problem. In particular, the monotone constraints required in \cite{Y15c} are proven to be unnecessary, as it will be implied by the action minimization property.

	     For each type of topological constraints, by considering the corresponding action minimization problem in a coordinate frame rotating around the vertical axis at a constant angular velocity $\om$, we find an entire family of simple choreographies (seen in the rotating frame), as $\om$ changes from $0$ to $N$. Such a family starts from one planar linear chain and ends at another (seen in the original non-rotating frame).  The action minimizer is collision-free, when $\om=0$ or $N$, but may contain collision for $0 < \om < N$. However all possible collisions must be binary and each collision solution is $C^0$ block-regularizable. 

	     Moreover for certain types of topological constraints, based on results from \cite{BT04} and \cite{CF09}, we show that when $\om$ belongs to some sub-intervals of $[0, N]$, the corresponding minimizer must be a rotating regular $N$-gon contained in the horizontal plane. As a result, this generalizes Marchal's $P_{12}$ family of the three body problem to arbitrary $N \ge 3$. 
	\end{abstract}
	
	\maketitle
	
\section{Introduction} \label{sec:intro}

In $N$-body problem, a \emph{simple choreography} is a special periodic solution, where all the masses chase each other on a single loop. We assume all masses are equal in the rest of the paper (it is still an open problem whether there exists a simple choreography with unequal masses, see \cite{C04}). Well known examples of simple choreographies including the \emph{rotating regular $N$-gon}, the \emph{Figure-Eight} of the three body and the \emph{Super-Eight} of the four body (\cite{BT04}, \cite{CM00}, \cite{Sh14}). All these examples belong to the family of \emph{planar linear chains}, where the corresponding loop is contained inside a two dimension plane and looks like a sequence of consecutive \emph{bubbles} \footnote{A bubble means a planar loop without any self intersection.} along a straight line and symmetric with respect to it.  For example, the rotating regular $N$-gon has one bubble, the Figure-Eight has two and the Super-Eight has three. This family was discovered numerically by Sim\'o \cite{Si00} (for pictures, see \cite{Si00} or \cite{CGMS02}). For the planar $N$-body problem, the author proved the existence of this family in \cite{Y15c}, by finding them as collision-free action minimizers under certain symmetric, topological and monotone constraints.  

The family of planar linear chains is important not only because they look interesting, but also because it may help us understand the global dynamics of the $N$-body problem (\cite{CFM05}, \cite{C08} and \cite{CF08}). To see this, let's consider the spatial $N$-body problem instead of the planar, and furthermore assume the coordinate frame is rotating around the vertical direction (the $z$-axis) at a uniformly angular velocity $\om$. In \cite{CF09} under certain symmetric constraints, using the Lyapunov center theorem and Weinstein-Moser theorem, Chenciner and F\'ejoz proved the local existence of Lyapunov families bifurcating in the vertical direction from some horizontal rotating regular $N$-gon (entirely contained in the $xy$-plane) with $\om$ as a parameter. 

While proving the local existence of these vertical Lyapunov families, they also studied the global existences of these families numerically in \cite{CF09}. Their numerical investigation found these families should exist for a large interval of $\om$ including $\om=0$, and as $\om$ decreases from some positive constant to zero, these families changes continuously from a rotating regular $N$-gon lying the $xy$-plane, to different planar linear chains entirely contained in the $yz$-plane. For example when $N=5$, F\'ejoz found one vertical Lyapunov family ends at a two loop chain, when $\om=0$ (see Figure \ref{fig:5body-p12}) and another one ends at a four loop chain, when $\om=0$ (see  Figure \ref{fig:5body-4loops}). More numerical results regarding families bifurcating from the regular rotating $N$-gon can be found in \cite{CF09} and \cite{Ca17}.  

\begin{figure}
  \centering
  \includegraphics[scale=0.45]{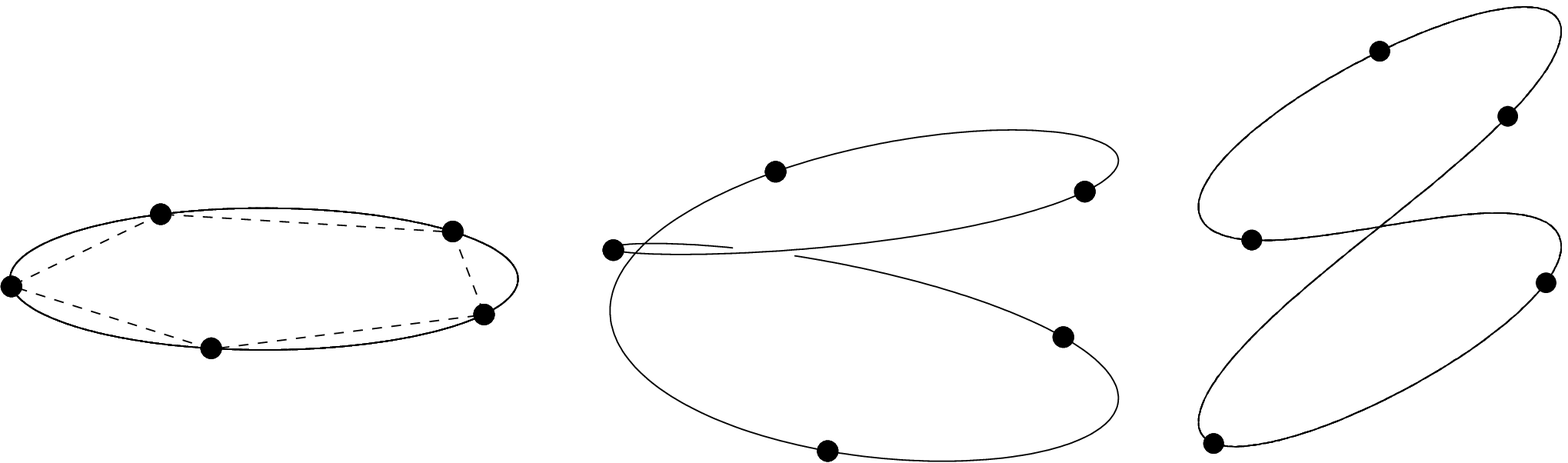}
  \caption{The 5 body figure eight}
  \label{fig:5body-p12}
\end{figure}

\begin{figure}
  \centering
  \includegraphics[scale=0.40]{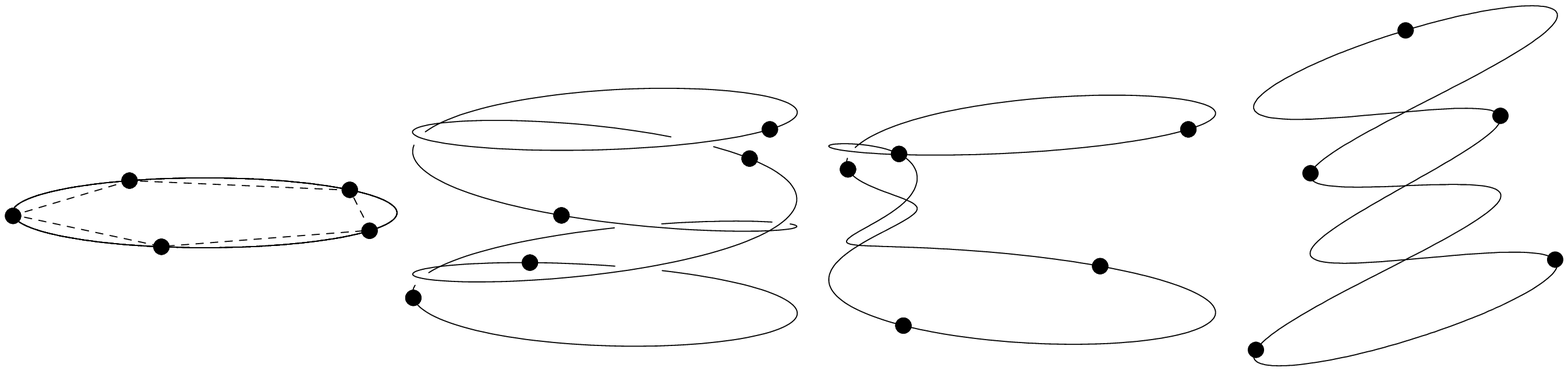}
  \caption{The 5 body 4 loops}
  \label{fig:5body-4loops}
\end{figure}

At the same time, by looking for minimizers of the action functional in uniform rotating frame under certain symmetric constraints, we can also find families of solutions of the $N$-body problem parameterized by the angular velocity $\om$ that belongs to a large interval, see \cite{BT04}, \cite{BFT08} and \cite{CF09}. In particular for $N=3$, using this approach Marchal (\cite{Mc00}) found a family of simple choreographies (seen in the rotating frame) known as the $P_{12}$ family: it starts with a figure eight solution contained in the $yz$-plane, when $\om=0$; as $\om$ increases, the two loops start to fold; when $\om$ reaches a small neighborhood of $2$, the two loops coincide with each other and the solution becomes a Lagrange relative equilibrium lying in the $xy$-plane and rotating twice within the given period, and moreover the size of the Lagrange relative equilibrium diverges as $\om$ approaches to $2$ (see Figure \ref{fig:p12} for pictures made by F\'ejoz).

\begin{figure}
  \centering
  \includegraphics[scale=0.45]{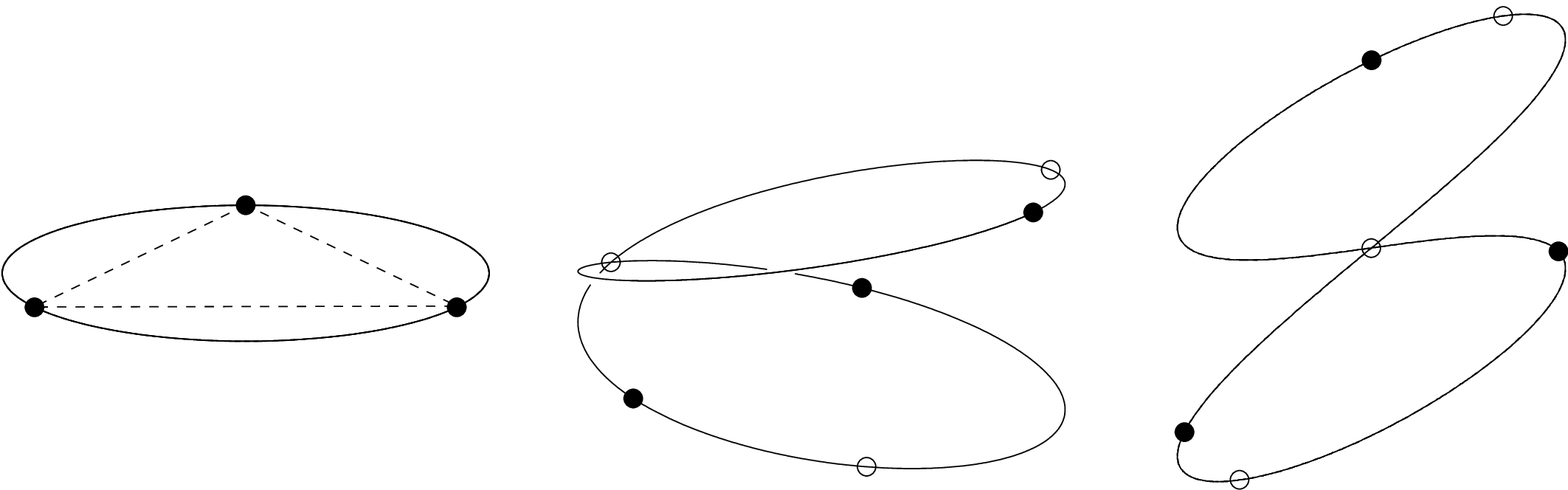}
  \caption{The $P_{12}$ family of three body}
  \label{fig:p12}
\end{figure}

\begin{rem} \begin{enumerate}
\item The problem of collision was not addressed by Marchal in \cite{Mc00}. Later in \cite{C02}, Chenciner showed they are all collision-free. 
\item In this paper, we will give a rigorous proof that when the angular velocity is zero, the minimizer from the $P_{12}$ family is a figure eight entirely contained in the $yz$-plane (which makes it precise the Figure-Eight proven by Chenciner and Montgomery in \cite{CM00}). Previously although widely believed and suggested by numerical results, this fact has never been proven (\cite{C02}, \cite{Fe06}).
\end{enumerate}
\end{rem}

When $N=3$, under the corresponding symmetric constraint, there is only one vertical Lyapunov family and likely it coincides with the $P_{12}$ family (although no proof is available). However for $N>3$, as pointed out in \cite{CF09}, there are more than one vertical Lyapunov families that all satisfies the same symmetric constraints, so the above action minimization approach can help us detect at most one of these families. To distinguish these families, one idea is to impose additional topological constraints to the minimization problem. This idea was used by the author in \cite{Y15c}, when we tried to establish the existence of the family of planar linear chains, as all the planar linear chains satisfying the same symmetric constraints and only one of them can be found as a minimizer under the corresponding symmetric constraints (in fact by a result of Barutello and Terracini \cite{BT04}, it is nothing but the rotating regular $N$-gon).

Using a similar idea,  in this paper we will look for action minimizers of the spatial $N$-body problem in uniform rotating frames with angular velocity $\om \in [0, N]$ under different combinations of symmetric and topological constraints. As a result, we find many families of simple choreographies (seen in the rotating frame) parameterized by $\om$. Moreover all these families start from one planar linear chain (when $\om=0$) and end at another one (when $\om=N$ and seen from the original non-rotating frame). 

In particular for certain choices of topological and symmetric constraints, combining our results with those from \cite{BT04} and \cite{CF09}, we will show when $\om$ is contained in some sub-interval of $[0, N]$, then the corresponding action minimizers must be a rotating regular $N$-gon entirely contained in the $xy$-plane. These families will include the Marchal's $P_{12}$ family, when $N=3$, as well as the families found numerically by F\'ejoz, when $N=5$, as we showed in Figure \ref{fig:5body-p12} and \ref{fig:5body-4loops}. Potentially they could also coincide with the vertical Lyapunov families discovered in \cite{CF09}. A rigorous proof of this will be interesting, but seems difficult to obtain.

To achieve our result, there are two main difficulties. First as usual is to show the action minimizers are collision-free. As it is well-known, the problem of ruling out collision is much more difficult when topological constraints are involved. In \cite{Y15c}, for the planar $N$-body problem, the problem was solved by imposing an additional monotone constraints (along a fixed direction) besides the symmetric and topological constraints. To make this idea work, at the beginning one needs to show a minimizer is not degenerate along the direction associated with the monotone constraints. This is not hard to do for the planar problem, as in this case the degenerate minimizer must be collinear and the symmetric and topological constraints ensures the existence of at least one isolated collision. Then one can reach a contradiction by some local deformation argument. However for the spatial problem, even in the degenerate case the masses are still allowed to move inside a two dimension plane, which gives them enough freedom to satisfy the symmetric and topological constraints and avoiding any collision. 

Second, we need to show for $\om=0$ or $N$, the corresponding action minimizer is actually planar (in the original non-rotating frame), i.e. it is contained in a two dimensional linear subspace of $\rr^3$. As we mentioned earlier, even for the $P_{12}$ family we are not sure it actually starts from the planar Figure-Eight. In general it is an interesting but difficult task to determine whether an action minimizer will spread out to the maximal possible dimensions, or only be contained inside a subspace with fewer dimension. As one can see by going to extra dimensions, we decrease the potential function, but increase the kinetic energy. However it is not so clear, which one will be the dominating term (see \cite{Ch07} by Chen).

It turns out the above two difficulties can more or less be resolved simultaneously. The key is to show that for an action minimizer under the particular symmetric constraints we are considering, it must satisfy certain monotone property (see Lemma \ref{lm: monotone fixed}). Consequently our result will show the monotone constraints introduced in \cite{Y15c} is in fact unnecessary. 

Unfortunately we are not able to prove the action minimizers is always collision-free (we suspect such a result does not always hold), except when $\om =0$ or $N$. However we are able to prove any possible collision must be an isolated binary collision (although there may be more than one binary collision at a collision moment) and the collision singularities are $C^0$ block-regularizable (see Definition \ref{dfn:block reg 1}, \ref{dfn:block reg 2} and \ref{dfn:block reg 3}). Moreover such a collision solution must be entirely contained in the $xy$-plane.

Our paper is organized as follows: in Section \ref{sec: main results}, the precise statements of our results will be given;  in Section \ref{sec: lemmas}, several technical lemmas that will be useful in ruling out collision will be introduced; in Section \ref{sec: fixed}, the family of planar linear chains will be proven as action minimizers of the spatial $N$-body problem (in the non-rotating frame) under certain symmetric and topological constraints; in Section \ref{sec: coercive}, we will study the existence of the action minimizers in uniform rotating frame with different angular velocities under certain symmetric and topological constraints; in Section \ref{sec: rotate}, the problem of collision regarding the action minimizers found in Section \ref{sec: coercive} will be investigated and it will be shown a minimizer is either collision-free or the collision singularities are $C^0$ block-regularizable.  

\section{the main results}  \label{sec: main results}
Let $q_i =(x_i, y_i, z_i) \in \rr^3$ represent the position of a point mass $m_i$, $i \in \N := \{0, \dots, N-1 \}$ and $q=(q_i)_{i \in \N} \in \rr^{3N}$. Without loss of generality, assume $m_i = 1$, $\forall i \in \N$. Under Newton's law of universal gravity, the motions of the masses satisfy the following equation
\begin{equation} \label{eq:nbody}
\ddot{q}_i = \frac{\partial}{\partial q_i} U(q) = -\sum_{j \in \N \setminus \{i\}} \frac{q_i-q_j}{|q_i -q_j|^3}, \quad \forall i \in \N,
\end{equation}
where $U(q)$ is the potential function, the negative potential energy, defined as below 
\begin{equation}
\label{eq: potential fun} U(q) = \sum_{0 \le i < j \le N-1 } \frac{1}{|q_i -q_j|}.
\end{equation}
Equation \eqref{eq:nbody} is the Euler-Lagrange equation of the action functional 
\begin{equation}
\label{eq:action} \A(q; T_1, T_2) = \int_{T_1}^{T_2} L(q(t), \qd(t)) \, dt, \;\; q \in H^1([T_1, T_2], \rr^{3N}),
\end{equation}
where $H^1([T_1, T_2], \rr^{3N})$ is the space of all Sobolev paths defined on the time interval $[T_1, T_2]$ and $L(q, \qd)$ is the Lagrangian
$$ L(q, \qd) = K(\qd)+U(q), \;\; K(\qd) =\ey \sum_{i \in \N} |\qd_i|^2 . $$
For simplicity, we set $\A(q; T)= \A(q; 0, T)$, for any $T>0$. 

If $q \in H^1([T_1, T_2], \rr^{3N})$ is a collision-free critical point of the action functional, then it is a smooth solution of equation \eqref{eq:nbody}. By \emph{collision-free}, we mean $q(t) \in \rr^{3N} \setminus \Delta$, for any $t \in [T_1, T_2]$, where $\Delta$ is the set of collision configurations
$$ \Delta :=\{q =(q_i)_{i \in \N} \in \rr^{3N}| \; q_{i_1} =q_{i_2}, \text{ for some } i_1 \ne i_2 \in \N \}. $$

We briefly recall the idea of imposing symmetric constraints following the notations from \cite{FT04}. Let $\Lmd= H^1(\rr / 2\pi\zz, \rr^{3N})$ be the space of $2\pi$-periodic Sobolev loops and $\hat{\Lmd}=H^1(\rr / 2\pi\zz, \rr^{3N} \setminus \Delta)$ the subset of collision-free loops. Given an arbitrary finite group $G$ with its action on the loop space $\Lmd$ defined as following
$$ g\big(q(t)\big)= \big(\rho(g)q_{\sg (g^{-1})(0)}, \dots, \rho(g) q_{\sg (g^{-1})(N-1)}\big) \big(\uptau(g^{-1})t \big), \; \forall g \in G, $$
where
\begin{enumerate}
 \item[(a).] $\uptau: G \to O(2)$ represents the action of $G$ on the time circle $\rr/ 2 \pi \zz$;
 \item[(b).] $\rho: G \to O(3)$ represents the action of $G$ on $\rr^3$;
 \item[(c).] $\sg: G \to \mc{S}_{\N}$ represents the action of $G$ on $\N$, where $\mc{S}_{\N}$ is the permutation group of $\N$. 
\end{enumerate}
Set $\Lmd^G = \{ q \in \Lmd|\; g(q(t)) = q(t), \; \forall g \in G \}$ as the space of \emph{$G$-equivariant loops} and $\hat{\Lmd}^G= \Lmd^G \cap \hat{\Lmd}$. As all masses are the same, the action functional $\A$ is invariant under the above group action. By Palais' symmetric principle \cite{Pa79}, a collision-free critical point of $\A$ in $\Lmd^{G}$ is a collision-free critical point of $\A$ in $\Lmd$ as well.

Let `Id' be the identity,  $\mf{R}_{xz}$ the reflection with respect to the $xz$-plane and $\mf{R}_x$ the rotation of $\pi$ around the $x$-axis ($\mf{R}$ with other sub-indices will be defined similarly). We define the symmetric constraints through the action of the dihedral group $ D_N : = \langle g, h |\; g^N= h^2 =1, (gh)^2 =1 \rangle$ by
\begin{equation}
 \label{eq: DN} \begin{cases}
 \uptau(g) t & = t-\frac{2\pi}{N}, \quad \rho(g) = \text{Id}, \quad \sg(g) =  (0,1, \dots, N-1); \\
 \uptau(h) t & = \frac{2\pi}{N}-t, \quad \rho(h) = \R_{xz}, \quad \sg(h) = \prod_{i=0}^{[\frac{N-1}{2}]} (i, N-1-i),
 \end{cases} 
\end{equation}
where $[k]$ represents the largest integer less than or equal to $k$, for any $k \in \rr$. 

The action of $g$ requires all the masses to follow the footstep of $m_0$, i.e.
\begin{equation}
\label{eq: simple choreography solution} q_i(t)= q_0(t + i \frac{2\pi}{N}), \;\; \forall t \in \rr, \; \forall i \in \N. 
\end{equation}
Hence each collision-free critical point of $\A$ in $\ldn$ will be a simple choreography. Meanwhile the action of $h$ implies
\begin{equation}
\label{eq: symm q0} q_0(2\pi-t)= \R_{xz}q_0(t), \;\; \forall t \in \rr. 
\end{equation}
This means the loop $q_0(\rr/ 2\pi \zz)$ is symmetric with respect to the $xz$-plane, and $q_0(t)$ belongs to the $xz$-plane, when $t=0$ or $\pi$. 

With the $D_N$-symmetry defined as above, there is a one-to-one correspondence between loops in $\ldn$ and paths $q \in H^1([0, \pi/N], \rr^{3N})$ satisfying the following conditions
\begin{equation} \label{eq: symmetric boundary moment 1}
\begin{cases}
q_i(0) = \R_{xz} q_{N-i}(0), \; &\forall 1 \le i \le [\frac{N-1}{2}], \\
q_i(\frac{\pi}{N}) = \R_{xz} q_{N-1-i}(\frac{\pi}{N}), \; &\forall 0 \le i \le [\frac{N-1}{2}]-1, \\
q_0(0) = \R_{xz} q_0(0), & \\
q_{[\frac{N-1}{2}]}(\frac{\pi}{N}) = \R_{xz} q_{[\frac{N-1}{2}]}(\frac{\pi}{N}).
\end{cases}
\end{equation}
The time interval $[0, \pi/N]$ will be called a \emph{fundamental domain} of $D_N$-equivariant loops. In the following, we will not distinguish between a loop from $\ldn$ and a path $q \in H^1([0, \pi/N], \rr^{3N})$ satisfying \eqref{eq: symmetric boundary moment 1}.

First we can ask if a global minimizer of $\A$ in $\ldn$ exists and if so, is it collision-free? By a result of Barutello and Terracini \cite{BFT08}, such a global minimizer exists under some additional coercive condition and it is collision-free. However it is nothing but the the rotating regular $N$-gon. As a result, if we want to find interesting and non-trivial solutions in $\ldn$, instead of global minimizers, we need to look for local minimizers. One approach is to impose extra topological constraints to the minimization problem and then look for action minimizers under the same $D_N$-symmetry but different topological constraints. 

Although the loop space $\hat{\Lmd}^{D_N}$ has infinitely many different connected components, one can not expect the action minimizers in each of these connect components to be collision-free, see \cite{Go77}, \cite{Ve01} and \cite{Mo02}, so we need to find the proper topological constraints. For the given $D_N$-symmetry, for any $q \in \hat{\Lmd}^{D_N}$, \eqref{eq: symmetric boundary moment 1} implies
\begin{equation} \label{eq: symmetric boundary moment 2}
\begin{cases}
y_i(0)=-y_{N-i}(0) \ne 0, \; & \forall 1 \le i \le [\frac{N-1}{2}], \\
y_i(\frac{\pi}{N}) = -y_{N-1-i}(\frac{\pi}{N}) \ne 0, \; & \forall 0 \le i \le [\frac{N-1}{2}]-1.  
\end{cases}
\end{equation}
This is equivalent to 
$$ y_0(\frac{k\pi}{N}) \ne 0, \; \forall 1 \le k \le N-1. $$
As a result, for any $\xi \in \Xi_N$, where 
\begin{equation}
\label{eq: Xi} \Xi_N:=\{\xi= (\xi_i)_{i=1}^{N-1}| \, \xi_i \in \{\pm 1\}, \; \forall i \},
\end{equation}
we can define
\begin{equation}
\hld_{\xi}: = \{q \in \hld |\; y_0(i/2) = \xi_i |y_0(i/2)|, \; \forall 1 \le i \le N-1\}. 
\end{equation}
Obviously if $q \in \hat{\Lmd}^{D_N}_\xi$ and $\qt \in \hat{\Lmd}^{D_N}_{\tilde{\xi}}$ with $\xi \ne \tilde{\xi}$, then there does not exist a continuous path of collision-free loops in $\hat{\Lmd}^{D_N}$ goes from $q$ to $\qt$. As $\hld_{\xi}$ is not closed, we will consider its weak closure in $\ldn$ with respect to the $H^1$ norm, which will be denoted by  $\ldn_\xi$. We say $q \in \ldn$ satisfies the \emph{$\xi$-topological constraints}, if $q \in \ldn_{\xi}$.  

We would like to show the action functional $\A$ has at least one minimizer in $\ldn_{\xi}$. For this it needs to be coercive in $\ldn_{\xi}$. Immediately we notice this is not true, as $\ldn_{\xi}$ is invariant under any linear translation along the $x$ or $z$-axis. One can solve this problem by fixing the center of mass at the origin (see \cite{FT04}). However for technical reason, we try to avoid such a strong assumption. Instead the following weaker conditions related to $q_0=(x_0, y_0, z_0) \in H^1(\rr / 2\pi \zz, \rr^3)$ will be required  
\begin{equation}
\label{eq: coercive x} [x_0] := \frac{1}{2\pi} \int_0^{2\pi} x_0(t) \,dt=0, 
\end{equation}
\begin{equation}
\label{eq: coercive y} [y_0]:= \frac{1}{2\pi} \int_0^{2\pi} y_0(t) \,dt=0,
\end{equation}
\begin{equation}
\label{eq: coercive z} [z_0] := \frac{1}{2\pi} \int_0^{2\pi} z_0(t) \,dt=0. 
\end{equation}
Notice that by \eqref{eq: symm q0}, \eqref{eq: coercive y} always holds for any $q \in \ldn$. Although the above conditions are given for the path of $m_0$, by \eqref{eq: simple choreography solution}, once they hold for the path of $m_0$, they will hold for the paths of other $m_i$'s as well. Now we will state our first result.

\begin{thm}
\label{thm: linear chain} For each $\xi \in \Xi_N$, the action functional $\A$ has at least one minimizer among all loops in $\ldn_{\xi}$ satisfying \eqref{eq: coercive x} and \eqref{eq: coercive z}, and each action minimizer $q$ is a collision-free simple choreography of \eqref{eq:nbody} satisfying the following properties.
\begin{enumerate}
\item[(a).] Either $\xd_0(t) \equiv 0$, $ \forall t \in \rr/ 2\pi \zz$, or $\xd_0(t)=0$, if and only if $t \in \{0, \pi\}$ and 
$$  \begin{cases}
\xd_0(t) >0 \; (\text{resp.} <0), & \; \text{ if } t \in (0, \pi), \\
\xd_0(t) <0 \; (\text{resp.} >0), & \; \text{ if } t \in (\pi, 2\pi). 
\end{cases} $$
\item[(b).] Either $\zd_0(t) \equiv 0, \forall t \in \rr/ 2\pi \zz$, or $\zd_0(t)=0$, if and only if $t \in \{0, \pi\}$ and 
$$  \begin{cases}
\zd_0(t) >0 \; (\text{resp.} <0), & \; \text{ if } t \in (0, \pi), \\
\zd_0(t) <0 \; (\text{resp.} >0), & \; \text{ if } t \in (\pi, 2\pi). 
\end{cases} $$
\item[(c).] $q_0(\rr/2\pi \zz)$ belongs to a fixed two dimensional plane, which is symmetric with respect to the $xz$-plane.
\end{enumerate}
\end{thm}

\begin{rem} \label{rem: strict monotone}
\begin{enumerate}
\item Notice that $\Lmd_{\xi}^{D_N}$ is invariant under an arbitrary rotation with respect to the $y$-axis, so each action minimizer in $\Lmd_{\xi}^{D_N}$ generates an entirely family of action minimizers that are identical to each other after a rotation with respect to the $y$-axis. 
\item Property (b) in the above theorem shows that if the simple choreography is not contained in a plane parallel to the $xy$-plane, then it must satisfy certain \textbf{strict monotone property} along the $z$-axis. In particular if $\zd_0(t) >0$, when $t \in (0, \pi)$, then the $D_N$-symmetry implies:  
\begin{equation}
\label{eq: symmetric monotone even} \text{when } N=2n, \;\; \begin{cases}
\zd_i(t) >0, \; & \forall t \in (0, \frac{\pi}{N}), \; \text{ if } i \in \{0, n\}, \\
\zd_i(t) >0, \; & \forall t \in [0, \frac{\pi}{N}], \; \text{ if } 1 \le i \le n-1, \\
\zd_i(t) <0, \; & \forall t \in [0, \frac{\pi}{N}], \; \text{ if } n+1 \le i \le N-1;
\end{cases}
\end{equation}
\begin{equation}
\label{eq: symmetric monotone odd} \text{when } N=2n+1, \;\; \begin{cases}
\zd_0(t) >0, \; & \forall t \in (0, \frac{\pi}{N}), \\
\zd_n(t)>0, \; & \forall t \in [0, \frac{\pi}{N}), \\
\zd_i(t) >0, \; & \forall t \in [0, \frac{\pi}{N}], \; \text{ if } 1 \le i \le n-1, \\
\zd_i(t) <0, \; & \forall t \in [0, \frac{\pi}{N}], \; \text{ if } n+1 \le i \le N-1.
\end{cases}
\end{equation}
Obviously property (a) implies  similar results along the $x$-axis. 
\end{enumerate} 
\end{rem}

Notice that $\ldn_{\xi}$ is invariant under any rotation with respect to the $y$-axis. Since the action functional is also invariant under these rotations, in fact there is an entirely family of minimizers. In particular one of them is entirely contained in the $xy$-plane, in which case $m_0$ starts from the $x$-axis at $t=0$, then keeps moving forward (or backward) along the direction of $x$-axis until it reaches the $x$-axis and turns around at $t=\pi$. Meanwhile $m_0$ must cross the $x$-axis at some moment $t \in (i \pi/N, (i+1) \pi/N)$, if $\xi_i \ne \xi_{i+1}$, for any $i = 1, \dots, N-2$. As a result, the corresponding solution looks like a sequence of loops placed continuous along the $x$-axis. Therefore it belongs to the family of linear chains (for numerical pictures see \cite{Si00} or \cite{CGMS02}). 

As we mentioned the existence of the family of linear chains has been established by the author in \cite{Y15c}. Nevertheless Theorem \ref{thm: linear chain} generalizes the result in \cite{Y15c} from a couple of aspects:
\begin{enumerate}
\item[(i).] in \cite{Y15c} the planar linear chains are obtained as collision-free minimizers of the planar $N$-body problem, while here they are show to be minimizers of the spatial $N$-body problem, so they are minimizers of a much larger family of loops;
\item[(ii).] in \cite{Y15c}, the linear chains are shown to be collision-free minimizers under additional monotone constraints, so even for the planar $N$-body problem, the above result shows the planar linear chains are collision-free minimizers of a larger family of loops.
\end{enumerate}

Theorem \ref{thm: linear chain} shows the existence of a $D_3$-symmetry Figure-Eight of the three body, when $\xi=(1,-1)$ and a $D_4$-symmetry Super-Eight of the four body, when $\xi=(1,-1,1)$. However as we recall the Figure-Eight of the three body in \cite{CM00} satisfies the $D_6$-symmetry and the Super-Eight of the four body in \cite{Sh14} satisfies the $D_4 \times \zz_2$-symmetry, where $\zz_2= \langle f | f^2 =1 \rangle$. Illuminated by these examples, we define the group $H_N := D_N \times \zz_2$, as the $\zz_2$ extension of $D_N$, with the action of $D_N$ defined as in \eqref{eq: DN} and the action of $\zz_2$ by: 
\begin{align}
 & \text{if } N=2n, \; \uptau(f)t =t, \;  \rho(f) = \R_x, \; \sg(f)= \prod_{i=0}^{n-1} (i, n+i) \label{eq: f even} \\
& \text{if } N=2n+1,  \; \begin{cases}  & \uptau(f)t =\frac{\pi}{N}-t, \; \rho(f) = \R_x, \\ 
&\sg(f)= \Big( \prod_{i=0}^{[\frac{n}{2}]} (i, n-i) \Big) \Big( \prod_{i=0}^{[\frac{n-1}{2}]}(n+1+i, 2n-i) \Big).
\end{cases}  
 \label{eq: f odd} 
\end{align}
Notice that as a group $H_N$ is isomorphic to $D_{2N}$, when $N=2n+1$. 

\begin{rem}
\label{rem: coercive z HN} Recall that if $q \in \ldn$, then $q_0(t)$ satisfies \eqref{eq: symm q0} and the corresponding loop $q_0(\rr/2\pi \zz)$ is invariant under the action of $\R_{xz}$. 

Meanwhile if $q \in \Lmd^{H_N}$, then besides \eqref{eq: symm q0}, $q_0(t)$ must satisfies 
\begin{equation} 
\label{eq: q0 HN symmetry} 
\begin{cases}
q_0(\pi-t)= \R_x q_0(t), & \text{ if } N \text{ is odd}; \\
q_0(t+\pi)= \R_x q_0(t), & \text{ if } N \text{ is even}. 
\end{cases}
\end{equation}
Hence $q_0(t)$ satisfies \eqref{eq: coercive z} and $q(\rr/2\pi \zz)$ is also invariant under the action of $\R_x$. 
\end{rem}
%\begin{equation}
 %\label{eq: q0 HN symmetry} x_0(0)= x_0(\pi),\;  z_0(0)= -z_0(\pi), \; z_0(\frac{\pi}{2})= z_0(\frac{3\pi}{2})=0. 
 %\end{equation} 

Given an arbitrary $\xi \in \Xi_N$, the $\xi$-topological constraints is not always compatible with the $H_N$-symmetry constraint, i.e., $\Lmd^{H_N}_\xi$ is a non-empty set if and only if 

\begin{equation}
 \label{eq: xi extra sym} |\xi_i - \xi_{N-i}| = \begin{cases}
 2, \;\; &\text{ if } N =2n+1, \\
 0, \;\; &\text{ if } N =2n, 
 \end{cases} \;\;\; \forall i \in \{1, \dots, [\frac{N-1}{2}] \}. 
 \end{equation}

\begin{thm} \label{thm: linear chain extra sym} For each $\xi \in \Xi_N$ satisfying \eqref{eq: xi extra sym}, $\A$ has at least one minimizer among all loops in $\Lmd^{H_N}_{\xi}$ satisfying \eqref{eq: coercive x}, and each minimizer $q$ is a collision-free simple choreography of \eqref{eq:nbody} satisfying all the properties listed in Theorem \ref{thm: linear chain}. Moreover $q_0(\rr/ 2\pi \zz)$ belongs to the $yz$-plane, i.e., $ x_0(t) \equiv 0, \;\; \forall t \in \rr/2\pi \zz. $
\end{thm}
\begin{rem}
\begin{enumerate}
\item As we mentioned in Remark \ref{rem: strict monotone}, a minimizer in $\Lmd^{D_N}_{\xi}$ generates an entire family of minimizers by rotating it with respect to the $y$-axis by an arbitrary angle. Meanwhile by the above Theorem a minimizer in $\Lmd^{H_N}_{\xi}$ must belong to the $yz$-plane, as if we rotate it with respect to the $y$-axis, it will not belong to $\Lmd^{H_N}_{\xi}$ anymore, except the rotating angle is $k\pi$ with $k \in \zz$. 
\item Although numerical results suggest the action minimizers obtained in Theorem \ref{thm: linear chain} under the $D_N$-symmetry should satisfy the $H_N$-symmetry as well (when $\xi$ satisfies \eqref{eq: xi extra sym}), no proof is available at this moment.
\item For $N=3$, the $H_3(\cong D_6)$-symmetry constraint is the same as the one used by Marchal in discovering the $P_{12}$ family (see \cite{Mc00} and \cite{CFM05}), so the above theorem shows, when $\om=0$, the corresponding minimizer in the $P_{12}$ family is entirely contained in the $yz$-plane and precisely the Figure-Eight solution of Chenciner and Montgomery. 

\end{enumerate}
\end{rem}

Now let's consider our problem in a coordinate frame rotating around the $z$-axis at a constant angular velocity $\om \in \rr$. To simplify notation, the $xy$-plane will be identified with the $1$-dim complex plane $\cc$ with $\J= \sqrt{-1}$, so $q_i= (x_i, y_i, z_i) \in \rr^3$ will also be written as 
$$ q_i=(\zt_i, z_i) \in \mathbb{C} \times \rr, \; \text{ where } \; \zt_i= x_i + \J y_i. $$ 
%\J: = \left( \begin{array}{ccc} 
%0 & -1 & 0 \\
%1 & 0 & 0 \\
%0 & 0 & 0 \\ 
%\end{array} \right). 

Given an arbitrary path $q \in H^1([T_1, T_2], \rr^{3N})$ in the rotating frame with angular velocity $\om$, the same path in the original non-rotating frame has the expression 
\begin{equation}
\label{eq: q om}  \ej q(t) = (e^{\J \om t} q_i(t))_{i \in \N} := ( (e^{\J \om t}\zt_i(t), z_i(t))_{i \in \N}, \;\; t \in [T_1, T_2].
\end{equation}
For a given angular velocity $\om$, we introduce the following action functional
$$ \ao(q; T_1, T_2):= \int_{T_1}^{T_2} L_{\om}(q,\qd) \,dt; \; \;\;  L_{\om}(q,\qd)= K_{\om}(q, \qd) + U(q)$$
where $U(q)$ is the potential function defined in \eqref{eq: potential fun} and
\begin{equation}
\label{eq: K om} K_{\om}(q, \qd):= \ey \sum_{i \in \N} (|\dot{\zt}_i + \J \om \zt_i|^2 + |\dot{z}_i|^2).
\end{equation}
Then 
$$ \ao(q; T_1, T_2) = \A(\ej q; T_1, T_2), $$
Moreover if $q(t)$ is a collision-free critical point of $\ao$, then it is a solution of 
\begin{equation} \label{eq: nbody rotate}
\begin{cases}
\ddot{\zt_i} & = \om^2 \zt_i - 2 \om \J \zt_i +\partial_{\zt_i} U(q), \\
\ddot{z}_i & = \partial_{z_i} U(q), 
\end{cases}
 \quad \quad \forall i \in \N.
\end{equation}
and the correspondingly $\ej q(t)$, is a solution of \eqref{eq:nbody}

For a given $G$-symmetry constraint, recall that a critical point of $\ao$ in $\Lmd^G$ is a critical point of $\ao$ in $\Lmd$, if $\ao$ is invariant under the action of $G$. For this to hold for any $\om$ with $K_{\om}$ defined as in \eqref{eq: K om}, we need the $z$-axis to be a \emph{rotation axis} with respect to the $G$-symmetry constraint, see \cite[Definition 2.11 and Lemma 2.15]{Fe06} for the details. We point out that for the $D_N$-symmetry constraint given before, the $z$-axis is a rotation axis for any $N$, but for the $H_N$-symmetry constraint, this is the case only when $N$ is odd.

%In particular $q(t)$ is a solution of \eqref{eq: nbody rotate} if and only if $\ej q(t)$ is a solution of \eqref{eq:nbody}. The idea of considering critical points of $\A$ under symmetric constraints, can be generalized to $\ao$ as well, because $\ao$ is also invariant under the action of $D_N$ (see \cite[Lemma 2.15]{Fe06}). 

Fix an arbitrary $\xi \in \Xi_N$, after the above explanation we may consider the minimization problem of $\ao$ in $\ldn_{\xi}$ for any $\om \in \rr$, and it seems natural to ask the following questions: 
{\em
\begin{enumerate}
\item[(I).] Does there always exist an action minimizer $q^{\om}$ of $\ao$ in $\ldn_{\xi}$ under some proper coercive conditions, for any $\om \in \rr$?
\item[(II).] If such an action minimizer $q^{\om}$ exists, will it be collision-free? If not, could it be regularized in a sense, so that it can still provide us useful information of nearby smooth solutions?
\item[(III).] If such an action minimizer $q^{\om}$ exists, will it be the unique one (after modulo the obvious symmetries)?
\item[(IV).] Will it be possible to establish some kind of continuity of this family $q^{\om}$ with respect to certain parameter, for example $\om$?
\end{enumerate} }

\begin{rem}
We notice that it will be enough to consider the above questions for $\om \in [0, N]$, as for other values of $\om$, the corresponding action minimization problem can be reduced to the previous cases after the following two steps: 

First, for any loop $q(t) \in \Lmd$, it belongs to  $\ldn_{\xi}$ if and only if the corresponding loop $e^{-2kN \J t}q(t)$ belongs to $\ldn_{\xi}$, for any $k \in \zz$. Moreover for any $t \in \rr$, $ U((e^{-2kN \J t}q(t))) = U(q(t))$ and 
$$ K_{\om +2kN}\left(e^{-2kN \J t}q(t), \frac{d(e^{-2kN \J t}q(t))}{dt} \right) = K_{\om}(q(t), \qd(t)).$$ 
This means $q(t)$ is a minimizer of $\ao$ in $\ldn_{\xi}$ if and only if $e^{-2kN \J t}q(t)$ is a minimizer of $\A_{\om + 2kN}$ in $\ldn_{\xi}$. Hence we only need to study the problem for $\om \in [-N, N]$. 

Second, a loop $q(t) \in \Lmd$ belongs to $\ldn_{\xi}$ if and only if the corresponding loop $\bar{q}(t)$ belongs to $\ldn_{-\xi}$, where 
$$ \bar{q}(t)= (\bar{q}_i(t))_{i \in \N} := (x_i(t)-\J y_i(t), z_i(t))_{i \in \N} \; \text{ and } \; -\xi: = (-\xi_i)_{i =1}^{N-1}. $$ 
Moreover for any $t \in \rr$, 
$$ K_{\om}(q(t), \qd(t))= K_{-\om}(\bar{q}(t), \dot{\bar{q}}(t)) \; \text{ and } \; U(q(t)) = U(\bar{q}(t)). $$
Again this implies $q(t) \in \ldn_{\xi}$ is a minimizer of $\ao$ if and only if $\bar{q}(t) \in \ldn_{-\xi}$ is a minimizer of $\A_{-\om}$ in $\ldn_{-\xi}$. Hence it will be enough for us to consider the minimization problem for $\om \in [0, N]$. 
\end{rem}

Regarding question (I), we have the next two theorems.  

\begin{thm}
\label{thm: coercive rotating non-integer} For any  $\om \in [0, N]$ and $\xi \in \Xi_N$, the following properties hold.
\begin{enumerate}
\item[(a).] If $\om \in [0, N] \setminus \zz$, the action functional $\ao$ has at least one minimizer among all loops in $\ldn_{\xi}$ satisfying \eqref{eq: coercive z}.
\item[(b).] If $\om \in \{0, N\}$, the action functional $\ao$ has at least one minimizer among all loops in $\ldn_{\xi}$ satisfying \eqref{eq: coercive x} and \eqref{eq: coercive z}.
\end{enumerate}
\end{thm}

When $k \in [1, N-1] \cap \zz$, $\A_k$ is not coercive among all loops in $\ldn$ even under the condition \eqref{eq: coercive z}, see \cite[Proposition 7]{BT04} or Theorem \ref{thm: coercive rotating integer} in the following. However when $k$ is coprime with $N$, i.e. $\text{gcd}(k, N)=1$, with the additional $\xi$-topological constraints, $\A_k$ may still be coercive among all loops in $\ldn_{\xi}$ satisfying \eqref{eq: coercive z}. To give a precise statement, let us consider the regular $N$-gons entirely contained in the $xy$-plane with the center of mass at the origin. The locations of masses are determined by the rule that starting from the vertex occupied by $m_i$, $m_{i+1}$ shall be placed on the $(k+1)$-th vertex (assuming the one occupied by $m_i$ is the first) along the clockwise direction with $m_0$ lying to the $x$-axis. Figure \ref{fig:N-gon} shows the locations of the masses in the five body problem for different $k$.

For each $k \in [1, N-1] \cap \zz$, up to a renormalization of size, there are precisely two such regular $N$-gon's, one with $m_0$ lying on the negative $x$-axis and the other with $m_0$ on the positive $x$-axis. In the following $\fn_k^{-}$ and $\fn_k^+$ will denote these two regular $N$-gon's with the size such that $e^{\J t} \fn_k^{\pm}$ are solutions of \eqref{eq:nbody}.

%As $m_i=1$, $\forall i \in \N$, for any $q \in \rr^{3N}$, its moment of inertia $I(q)=|q|^2= \sum_{i \in \N} (x_i^2+y_i^2+z_i^2)$. For $\fn_k^{\pm}$, since $I(\fn_k^{\pm})=1$, the distance of each $m_i$ to the origin is $\frac{1}{\sqrt{N}}$. 
%\end{rem}

%We denote these regular $N$-gon's by $\fn_k$, when the distance of each mass to the origin is $\frac{1}{\sqrt{N}}$, and $\fn_k^{\infty}$, when this distance is infinity. 

\begin{figure}
  \centering
  \includegraphics[scale=0.60]{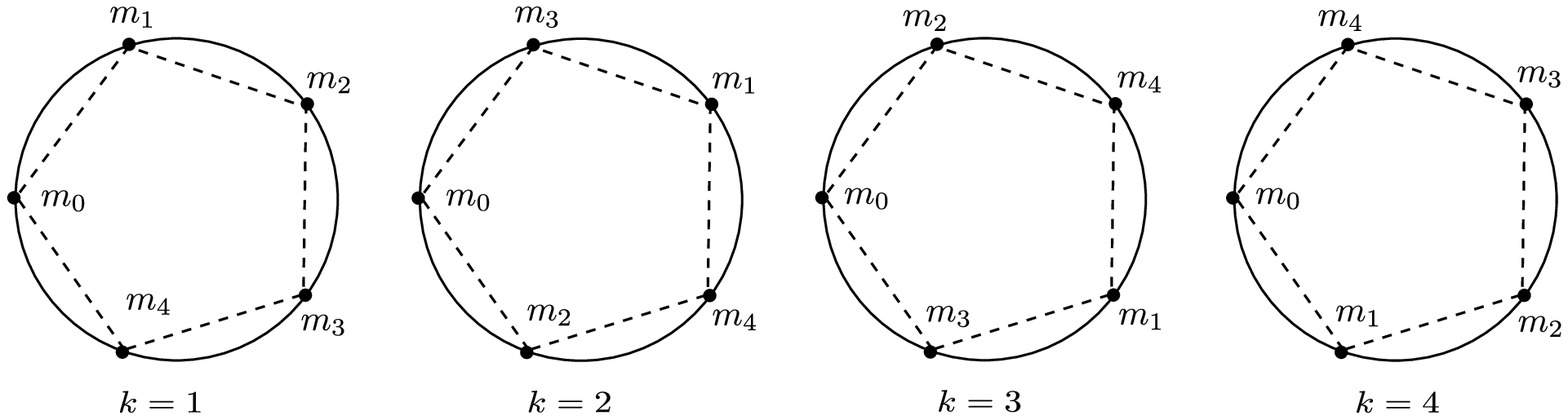}
  \caption{The regular 5-gon's}
  \label{fig:N-gon}
\end{figure}

%such that $e^{\J t}\fn_k$ is a periodic solution of \eqref{eq:nbody} with $2\pi$ as the minimal period,

%such that each loop $\lmd e^{-\J kt} \fn_k \in \ldn_{\xi(\fn_k)}$, for any $\lmd >0$. Clearly all these loops satisfy \eqref{eq: coercive x} and \eqref{eq: coercive z}. Notice that $\|\lmd e^{-\J kt} \fn_k\|_{\hnm}$ goes to infinity, when $\lmd$ goes to infinity. Meanwhile $K_k(\lmd e^{-\J kt} \fn_k, \frac{d(\lmd e^{-\J kt} \fn_k)}{dt}) =0$, for any $\lmd$, and $U(\lmd e^{-\J kt} \fn_k)=\lmd^{-1} U(\fn_k)$ goes to zero, as $\lmd$ goes to infinity. Therefore $\A_{k}(\lmd e^{-\J kt} \fn_k; 2\pi)$ goes to zero, when $\lmd$ goes to infinity, which violates the conditions of coercivity. 

%\begin{equation}
%\Xi_N^* = \{\xi(\fn_k):\;  k \in [1, N-1] \cap \zz \text{ satisfying } \text{gcd}(k, N)=1 \},
%\end{equation}
%Our next result gives an answer to question (I). 

\begin{rem}
\label{rem: rotating N-gon} We point out that for any $k \in [1, N-1]$ satisfying $\text{gcd}(k, N)=1$ and $\lmd >0$, $\lmd e^{-\J k t}\fn^{\pm}_k$ is a $2\pi$-periodic loop satisfying the $D_N$-symmetry constraint. In fact it also satisfies the $H_N$-symmetry constraint, when $N$ is odd, but not when $N$ is even . As a result, there is a unique $\xi(\fn^{\pm}_k) \in \Xi_N$, such that  $\lmd e^{-\J kt} \fn_k^{\pm} \in \ldn_{\xi(\fn_k^{\pm})}$.
\end{rem}

%For any $k \in [1, N-1] \cap \zz$ satisfying $\text{gcd}(k, N)=1$, since $\A_k(\lmd e^{-\J kt} \fn_k^{\pm}; 2\pi)$ goes to zero, as $\lmd$ goes to infinity, $\A_k$ is not coercive among all loops in $\ldn_{\xi(\fn_k^{\pm}))}$ satisfying \eqref{eq: coercive z}. However we can still prove the following result. 

%When $\om =k$, for some $k \in [1, N-1] \cap \zz$ satisfying $\text{gcd}(k,N)=1$, property (d) in the above theorem clearly shows $\ao$ is not coercive among all loops in $\ldn_{\xi(\fn_k)}$ satisfying \eqref{eq: coercive z}. However it also tells us that there always exists a minimizing sequence, which converges to $e^{-\J kt} \fn_k^{\infty}$. Notice that in the non-rotating frame, this is just a regular $N$-gon sitting still at infinity without any motion.

%As a result, if we add $e^{-\J kt}\fn^{\infty}_k$ into $\ldn_{\xi(\fn_k)}$, then they will be the action minimizers of $\A_k$. In the rest of the paper, when we talk about action minimizer of $\A_k$ in these cases, we will be referring to these infinitely large rotating $N$-gon's. 

\begin{thm}
\label{thm: coercive rotating integer} For any $\xi \in \Xi_N$ and $k \in [1, N-1] \cap \zz$ satisfying $\text{gcd}(k, N)=1$, the following properties hold.
\begin{enumerate}
\item[(a).] If $\xi \ne \xi(\fn_k^{\pm})$, $\A_k$ has at least one minimizer among all loops in $\ldn_{\xi}$ satisfying \eqref{eq: coercive z}. 
\item[(b).] If $\xi = \xi(\fn_k^{\pm})$, $\A_k$ does not have a minimizer among all loops in $\ldn_{\xi}$ satisfying \eqref{eq: coercive z}. In particular, $ \inf\{ \A_k (q)| \; q \in \ldn_{\xi(\fn_k^{\pm})} \text{ satisfying } \eqref{eq: coercive z} \}= 0.$

\item[(c).] If $\xi = \xi(\fn_k^{\pm})$, let $\qn \in \ldn_{\xi(\fn_k^{\pm})}$ satisfying \eqref{eq: coercive z} be an arbitrary minimizing sequence of $\A_k$, i.e. $\lim_{n \to \infty} \A_k (\qn; 2\pi)=0$, then $|q^n_i(t)| \to \infty$ uniformly on $t \in \rr$, for any $i \in \N$, and after passing to a subsequence, $\frac{\qn(t)}{|\qn(t)|}$ converges uniformly to $e^{-\J k t} \frac{\fn_{k}^{\pm}}{|\fn_k^{\pm}|}$ on $t \in \rr$. %Moreover for any subsequence of $\frac{\qn(t)}{|\qn(t)|}$, whenever it converges, it must converge uniformly to $e^{-\J k t} \fn_{k}$ on $t \in \rr$. 
\end{enumerate}
\end{thm}

%$$ \inf\{\A_k(q): \; q \in \ldn_{\xi(\fn_k)}\}=0, \; \text{ if } \xi(\fn_k) \in \Xi_N^*, $$
%and if $\qn \in \ldn_{\xi(\fn_k)}$ is an action minimizing sequence of $\A_k$, i.e. $\lim_{n \to \infty} \A_k(\qn)= 0$, then after normalizing their sizes and passing to a subsequence, it will converge to $e^{-\J kt} \fn_k$ (independent of the choice of subsequence). Therefore when $\xi= \xi(\fn_k) \in \Xi_N^*$ and $\om=k$, we will add the infinite large $N$-gon $\fn_k$ rotating at an angular velocity $-k$ into $\ldn_{\xi}$, and it will be the minimizer of $\A_k$ in $\ldn_{\xi(\fn_k)}$. In fact if we go back to the original none rotating frame, this just corresponds to $N$ pointed masses sitting still at infinity and with a normalized shape of $\fn_k$.

With the families of minimizers obtained by the above theorems, the next obviously is to see if they are collision-free.\begin{thm}
\label{thm: linear chain rotate} For any $\om \in [0, N]$ and $\xi \in \Xi_N$, if $\qo$ is an action minimizer of $\ao$ among all loops in $\ldn_{\xi}$ satisfying \eqref{eq: coercive z} (and \eqref{eq: coercive x}, if $\om \in \{0, N\}$), then the following properties hold. 
\begin{enumerate}
\item[(a).] If $\om \in \{0, N\}$, $\qo$ is a collision-free solution of \eqref{eq: nbody rotate}. Moreover when $\om =N$, in the non-rotating frame, the corresponding $e^{\J Nt}\qo$ is a collision-free minimizer of $\A$ among all loops in $\ldn_{\xi^*}$ satisfying \eqref{eq: coercive x} and \eqref{eq: coercive z}, where 
\begin{equation}
 \label{eq: xi star} \xi^*_i= \begin{cases}
 -\xi_i, \; & \text{ if } i \text{ is odd}, \\
 \xi_i, \; & \text{ if } i \text{ is even}.
 \end{cases}
 \end{equation} 

\item[(b).] If $\om \in (0, N)$, then either $\qo$ is a collision-free solution of \eqref{eq: nbody rotate} or the set of collision moments $\Delta^{-1}(\qo) :=\{ t \in \rr: \; \qo(t) \in \Delta \}$ is non-empty. In the latter case, $\Delta^{-1}(\qo) \subset \{ t = \ell \pi/N: \; \ell \in \zz\},$ and $\qo(t)$, $t \in \rr \setminus \Delta^{-1}(\qo)$, is a solution of \eqref{eq: nbody rotate}.

% Moreover $\qo(t)$ satisfies \eqref{eq: nbody rotate}, for any $t \in \rr \setminus \Delta^-(\qo)$ and 
%$$ K_{\om}(\qo(t), \qd^{\om}(t))-U(\qo(t)) \equiv \text{Constant}, \;\; \forall t \in \rr \setminus \Delta^-(\qo). $$

\item[(c).] If $\om \in (0, N)$ and $\Delta^{-1}(\qo) \ne \emptyset$, then $\qo(t)$ belongs to the $xy$-plane, for all $t \in \rr / 2\pi \zz$, and for any $t \in \Delta^{-1}(\qo)$, $\qo(t)$ can only have binary collisions. In particular, when $t =0$ or $\pi/N$, $\qo(t)$ can only have binary collisions between $m_j$ and $m_k$ satisfying
\begin{equation} \label{eq; possible binary collision}
 \{j, k\}= \begin{cases}
 \{i, N-i\}, \text{ for some } i \in \{1, \dots, [\frac{N-1}{2}]\}, \; & \text{ if } t=0, \\
 \{i, N-1-i \}, \text{ for some } i \in \{0, \dots, [\frac{N}{2}]-1\}, \; & \text{ if } t=\pi/N. 
 \end{cases}
\end{equation}
For other collision moments, the possible pairs of binary collisions can be determined through the $D_N$-symmetry. 

\end{enumerate}
\end{thm}

\begin{rem}
By Theorem \ref{thm: linear chain}, when $\om=0$, the minimizer $\qo(t)$ must be a planar linear chain contained in a plane symmetric with respect to the $xz$-plane, and when $\om=N$, so is $e^{\J Nt}\qo(t)$ in the original non-rotating frame, although most likely the corresponding $\qo(t)$ in the rotating frame is not contained in any fixed two dimensional plane.
\end{rem}

Here we are not able to show $\qo$ is always collision-free, when $\om \in (0, N)$. In fact, we think a general result like this does not hold and make the following conjecture.
\begin{conj}
There exist $N > 3$, $\om \in (0, N)$ and $\xi \in \Xi_N$, such that if $\qo$ is an action minimizer of $\ao$ among all loops in $\ldn_{\xi}$ satisfying \eqref{eq: coercive z}, then it must contain at least one collision.  
\end{conj}

Meanwhile property (c) in the above theorem has the obvious corollary, where the condition may be relatively easy to verify using rigorous numerical method. 
\begin{cor}
For any $\om \in (0, N)$ and $\xi \in \Xi_N$, let $\qo$ be an action minimizer of $\ao$ among all loops in $\ldn_{\xi}$ satisfying \eqref{eq: coercive z}, then if there are two different moment $t_0 \ne t_1$, such that $z_0(t_0) \ne z_1(t_1)$, then $\qo$ must be a collision-free solution of \eqref{eq: nbody rotate}. 
\end{cor}

Facing the possibility of having collision in the corresponding action minimizers, the next thing we can hope is to show they are regularizable in some sense, so that these collision solutions will still provide us useful informations of the nearby dynamics of the $N$-body problem, see the comments by Montgomery in \cite{Mo02}. Results about regularizable collision solutions that are found as action minimizers can be found in \cite{Sh11} and \cite{MVV13}. 

By Theorem \ref{thm: linear chain rotate}, the action minimizers can only have binary collisions at a collision moment. Although there may be just one binary collision or more than one, and in the former case better result can be obtained using Kustaanheimo-Stiefel regularization \cite{KS65}, we will not discuss them separately, as we can not rule out the latter and it contains the former as a special case. By abuse of notation, we will call both cases \emph{simultaneous binary collisions}. The regularization of simultaneous binary collisions has been studied by several authors, see \cite{SL92}, \cite{EB96} and \cite{MS00}.

%There are different types of regularization in celestial mechanics. From a dynamical system point of view, we will use the \emph{block-regularization} introduced by Easton in \cite{Es71}(this name was in fact given by McGehee in \cite{MG81}, as Easton called it \emph{regularization by surgery}). Here we following the definition given in \cite{EB96}. 

The result will work in our case is from \cite{EB96}. To explain it, let's consider the $N$-body problem in the non-rotating frame. Recall that associated to the Lagrangian $L(q, \qd)$, we have the Hamiltonian 
$$ H(q,p)=\langle \frac{\partial L}{\partial \qd}, p \rangle -L(q, \qd), \; \text{ where } p= \frac{\partial L}{\partial \qd}(q, \qd)$$
with the corresponding Hamiltonian vector field
\begin{equation}
\label{eq Ham vector field} \qd = \partial_{p}H(q, p); \;\;\; \dot{p} = -\partial_q H(q,p). 
\end{equation}
Since $m_i=1$, for all $i \in \N$, $q(t)$ is a solution of \eqref{eq:nbody} if and only if $(q, \qd)(t)$ is an orbit of \eqref{eq Ham vector field}.

Given a $q^- \in C^2((t_0-2\dl, t_0), \rr^{3N}) \cap C^0((t_0-2\dl, t_0], \rr^{3N})$ (or $q^+ \in C^2((t_0, t_0+2\dl_0), \rr^{3N}) \cap  C^0([t_0, t_0+2\dl_0), \rr^{3N})$), for some $\dl >0$, we say $(q^-, \qd^-)(t)$ (or $(q^+, \qd^+)(t)$), is  \emph{a collision orbit} (or \emph{an ejection orbit}) of \eqref{eq Ham vector field}, if $q^{\pm}(t) \notin \Delta$ and $(q^{\pm},\qd^{\pm})(t)$ satisfies \eqref{eq Ham vector field}, for any $t \ne t_0$, and $q^{\pm}(t_0) \in \Delta$. 

%(given by $q^+ \in C^2((t_0, t_0+2\dl_0), \rr^{3N}) \cap  C^0([t_0, t_0+2\dl_0), \rr^{3N})$

\begin{dfn} \label{dfn:block reg 1} We say a collision orbit $(q^-, \qd^-)(t)$ with $q^-(t_0) \in \Delta$ is $C^k$ \textbf{block-regularizable}, $k \ge 0$, if there is a unique ejection orbit $(q^+, \qd^+)(t)$, such that $q^+(t_0) =q^-(t_0)$, and there are two $(6N-1)$-dim cross sections $\Sigma^{\pm}$ both transversal to the vector field \eqref{eq Ham vector field}, such that $(q^{\pm}, \qd^{\pm})(t \pm \dl) \in \Sigma^{\pm}$ and the map 
$$ \phi: \Sigma^- \setminus \{(q^-, \qd^-)\}  \to \Sigma^+ \setminus \{ (q^+, \qd^+)\} $$ 
induced by the Hamiltonian flow of \eqref{eq Ham vector field} is a $C^k$ diffeomorphism, and moreover by defining $\phi((q^-, \qd^-))= (q^+, \qd^+)$, we can extend $\phi$ to a $C^k$ diffeomorphism from $\Sigma^-$ to $\Sigma^+$. 

We say a collision singularity $q^* \in \Delta$ is $C^k$ \textbf{block-regularizable}, if any collision orbit $(q^-, \qd^-)(t)$ with $q^-(t_0)=q^*$ is $C^k$ block-regularizable. 
\end{dfn}

The following result was proven in \cite[Corollary G]{EB96}. 
\begin{prop}
\label{prop: sbc reg} Any simultaneous binary collisions in the spatial $N$-body problem is $C^0$ block-regularizable. 
\end{prop}

Notice that in our approach when an action minimizer has a collision moment, the collision and ejection orbits are given a priori. In order to prove the action minimizer is block regularizable, we need to show they form the unique pair of collision and ejection orbits associated to each other. For this, we introduce the following definitions. 

\begin{dfn} \label{dfn:block reg 2}
Given a $q \in C^2((t_0-2\dl, t_0) \cup (t_0, t_0+2\dl_0), \rr^{3N}) \cap C^0((t_0-2\dl, t_0 +2\dl), \rr^{3N})$, we say $(q, \qd)(t)$ is a \emph{collision-ejection orbit} of \eqref{eq Ham vector field}, if $q(t) \notin \Delta$ and $(q, \qd)(t)$ satisfies \eqref{eq Ham vector field}, for any $t \ne t_0$, and $q(t_0) \in \Delta$. We say such a collision-ejection orbit is $C^k$ \textbf{block-regularizable}, $k \ge 0$, if $(q,\qd)(t), t \in (t_0-2\dl, t_0)$ is a $C^k$ block-regularizable collision orbit, and $(q,\qd)(t), t \in (t_0, t_0 +2\dl)$ is the unique ejection orbit associated to it. 
\end{dfn}

%where in Corollary G, the author proved all simultaneous binary collisions are \emph{$C^0$ block-regularizable}. \emph{Block-regularization} was introduced by Easton in \cite{Es71}(although the name was given by McGehee in \cite{MG81}, as Easton called it \emph{regularization by surgery}). 
\begin{dfn} \label{dfn:block reg 3}
Given a $q \in C^0(\rr, \rr^{3N})$,  we say it is a \textbf{collision solution} of \eqref{eq:nbody}, if $\Delta^{-1}(q)=\{t \in \rr: q(t) \in \Delta \}$ is a non-empty and isolated subset of $\rr$, and $q \in C^2(\rr \setminus \Delta^{-1}(q), \rr^{3N})  q(t)$ satisfies equation \eqref{eq:nbody}. Such a collision solution $q(t)$ is called $C^k$ \textbf{block-regularizable}, $k \ge 0$, if for each $t_0 \in \Delta^{-1}(q)$, there is a $\dl>0$, such that the corresponding collision-ejection orbit $(q,\qd)(t)$, $t \in (t_0-2\dl, t_0+2\dl)$ is $C^k$ block-regularizable. 
\end{dfn}

After the above explanation, we can see even if an action minimizer is just a collision solution, as long as they are  $C^k$ block-regularizable, it still carries useful information for the nearby dynamics of the $N$-body problem. For this reason, we will prove the following result. 

\begin{thm}
\label{thm: regularization} For any $\om \in (0, N)$ and $\xi \in \Xi_N$, let $\qo$ be an action minimizer of $\ao$ among all loops in $\ldn_{\xi}$ satisfying \eqref{eq: coercive z}, if $\qo$ is not collision-free, then in the non-rotating frame, the corresponding $e^{\J \om t}\qo(t)$ is a $C^0$ block-regularizable collision solution of \eqref{eq:nbody}. 
\end{thm}

\begin{rem}
It will be interesting if one can improve the regularity of the block-regularization in the above theorem. Our proof is based on results in \cite{EB96}, where only $C^0$ regularity is obtained. In \cite{MS00}, simultaneous binary collisions are also shown to be $C^k$ block-regularizable, for $k=8/3$, however it only applies to some special cases and does not seem to work here.  
\end{rem}

Like in the case of non-rotating frame, we may consider the corresponding minimization problem under the $H_N$-symmetry constraint, when $N$ is odd, as under the $H_N$-symmetry constraint, $\ao$ is invariant under the action of $H_N$ only when $N$ is odd. 

\begin{thm}
\label{thm: lc rotate extra sym} When $N$ is odd, for any $\om \in [0, N]$ and $\xi \in \Xi_N$ satisfying \eqref{eq: xi extra sym}, if $\qo$ is an action minimizer of $\ao$ among all loops in $\Lmd^{H_N}_{\xi}$ satisfying \eqref{eq: coercive z} (and \eqref{eq: coercive x}, if $\om \in \{0, N\}$), then it satisfies all the properties in Theorem \ref{thm: linear chain rotate} and \ref{thm: regularization}.
\end{thm}
\begin{rem}
By Theorem \ref{thm: linear chain extra sym}, when $\om=0$, the minimizer $\qo(t)$ is a planar linear chain contained in the $yz$-plane, and when $\om=N$,  so is $e^{\J Nt}\qo(t)$ in the original non-rotating frame.
\end{rem}

The previous three theorems more or less give us an answer to question (II). Compare to question (I) and (II), question (III) and (IV) are much more difficult. Despite of this, there are some partial results available. 

Recall that for any $k \in [1, N-1] \cap \zz$ satisfying $\text{gcd}(k,N)=1$, $e^{\J t}\fn_k^{\pm}$ is a solution of \eqref{eq:nbody} (with period $2\pi$). Due to the homogeneity of potential, for any $\lmd \in \rr$, $ |\lmd|^{-\frac{2}{3}}e^{\J \lmd t} \fn_k^{\pm}$, is a solution of \eqref{eq:nbody} as well. Then in the rotating frame with angular velocity $\om$,  
$$ \cn_{k, \om}^{\pm}(t):= |\om-k|^{-\frac{2}{3}} e^{-\J k t} \fn_k^{\pm}, \; t \in \rr, $$
is a periodic solution of \eqref{eq: nbody rotate} with minimal period $2\pi/k$. Meanwhile by Remark \ref{rem: rotating N-gon}, there is a unique $\xi(\fn_k^{\pm}) \in \Xi_N$, such that $\cn_{k, \om}^{\pm}(t) \in \ldn_{\xi(\fn^{\pm}_k)}$. Moreover when $N$ is odd, $\cn_{k, \om}^{\pm}(t)$ also belongs to $\Lmd^{H_N}_{\xi(\fn^{\pm}_k)}$. Combining the above theorems and results from Barutello and Terracini \cite{BT04} or Chenciner and F\'ejoz in \cite{CF09}, immediately we have the following corollary.
\begin{cor}
\label{cor: rotate N-gon} For any $k \in [1, N-1] \cap \zz$ satisfying $\text{gcd}(k, N)=1$, there is a constant $\dl(k, N) >0$, such that when $\om \in (k-\dl(k, N), k + \dl(k, N)) \setminus \{k\}$, $\cn_{k, \om}^{\pm}(t)$ is the unique action minimizer of $\ao$ among all loops in $\ldn_{\xi(\fn_k^{\pm})}$ satisfying \eqref{eq: coercive z}, and $\cn_{k, \om}^{\pm}(t)$ depends continuously on $\om$. 

When $N$ is odd, the same result holds under the $H_N$-symmetry constraint.  
\end{cor}
\begin{rem}
Notice that as $\om$ goes to $k$, the size of the regular $N$-gon goes to infinity. In some sense the corresponding minimizer for $\om =k$ can be seen as an infinite large regular $N$-gon without any motion in the original non-rotating frame. 
\end{rem}

We finish this section with some explanation of the above corollary for the $5$-body problem under the $H_5$-symmetry constraint. When $N=5$ and $k=2$, $\xi(\fn_2^-)=(1, 1, -1, -1)$. The corresponding family of minimizers $\qo(t)$ starting from a figure eight in the $yz$-plane when $\om=0$, as $\om$ increases the two loops in the eight begin to fold and when $\om \in  (2-\dl(2, 5), 2+ \dl(2, 5)) \setminus \{2\}$, $\qo(t) = \cn_{2, \om}^-(t)$, which are rotating 5-gons entirely contained in the $xy$-plane with its size goes to infinity, as $\om$ approaches to $2$. When $\om>2+ \dl(2, 5)$, the masses should not be able to stay on the $xy$-plane all the time. In particular, when $\om =5$, in the original non-rotating frame $e^{\J \om t}\qo(t)$ is a three-loop planar linear chain (a super eight) contained in the $yz$-plane satisfying the $\xi^*= (-1, 1, 1, -1)$-topological constraint. When $N=5$ and $k=4$, $\xi(\fn_4^-)=(1, -1, 1, -1)$. Things are similar as above, except when $\om=0$, $\qo(t)$ is the four-loop linear chain and when $\om =5$, in the original non-rotating frame $e^{\J \om t}\qo(t)$ is a regular rotating $5$-gon (a loop) contained in the $yz$-plane satisfying the $\xi^*=(-1, -1, -1, -1)$-topological constraint.

%The second result in the above corollary only holds for odd $N$, as $\cn_{k, \om}^{\pm}(t)$ satisfies the $H_N$-symmetry constraint, when $N$ is odd, but not when $N$ is even, see Remark \ref{rem: rotating N-gon}. Moreover a refined formula that can be used to estimate $\dl(k, N)$ under different types of symmetric constraints were obtained by Chenciner and F\'ejoz in \cite{CF09}. 

%We are not able to rule out collision completely for the action minimizers obtained above. Actually we suspect such a result may not exist, after all there are known examples that the action minimizers indeed contain collisions under topological constraints, see \cite{Ve01} and \cite{Mo02}. If we drop the topological constraints, then we are able to show the corresponding action minimizers are collision-free. However it is less likely such a family can hold some global continuity, in particular with respect to the parameter $\om$. Meanwhile although the families of solutions we obtained may contain collision, but they can only be binary collisions and are regularizable, and more important it is more likely the answers to question (III) and (IV) will be positive for these families. 

%\begin{cor}
%\label{cor: P12} For odd $N$ and $\om \in [0, N]$, there exists a family of minimizers $\qo$ of $\ao$ in $\Lmd^{H_N}_\xi$ parametrized by $\om$, which are always collision-free solution ****************
%\end{cor}

%---------------------------------------------------

\section{Technical lemmas} \label{sec: lemmas}

In this section, we collect several deformation lemmas that can be used to decrease the action value of collision paths. In the case of non-rotating frame, these results are proven in a series of papers by the author (\cite{Y15b}, \cite{Y15c} and \cite{Y16s}). Here we generalize them to the case of uniformly rotating frame. Throughout this section, we assume $q(t), t \in [0, T]$ a collision solution of \eqref{eq: nbody rotate}, which is collision-free, for any $t \in (0, T)$, and contains at least one collision at the moment $t=0$. %Similar results can be obtained when there is at a least one collision at the moment $t=T$. However we will only give the details for the case that there is at least one collision at the moment $t=0$. The case for $t=T$ is similar and left to the readers. 

Given an $\I \subset \N$, we say $q(t)$ has an $\I$-cluster collision at the moment $t=t_0$, if 
$$ \forall i \in \I, \;\; \begin{cases}
q_i(t_0) = q_j(t_0), \; & \text{ if } j \in \I \setminus \{i\}; \\
q_i(t_0) \ne q_j(t_0), \; & \text{ if } j \in \N \setminus \I. 
\end{cases}
$$ 
Let's assume $q(0)$ has an $\I$-cluster collision. In the first half of the section, we further assume $|\I|=2$ ($|\I|$ represents the cardinality of the set $\I$). Without loss of generality, let's say $\I= \{j, k \}$. Let $q_c(t)= \frac{q_j(t) + q_k(t)}{2}$ be the center of mass of $m_j$ and $m_k$, and
\begin{equation}
 \label{eqn:rel} \mf{q}_i(t) = (\xf_i(t), \yf_i(t), \zf_i(t)) = q_i(t) - q_c(t), \; \forall i \in \I,
\end{equation}
the relative position of $m_i$ with respect to the center of mass of $m_j$ and $m_k$. Introducing the spherical coordinates $(r, \phi, \tht)$ of $\rr^3$  with $r \ge 0, \phi \in [0, \pi]$ and $\tht \in \rr$, then
\begin{equation} \label{eqn:sph co}
\xf_i  = r_i \sin \phi_i \cos \tht_i, \;\; \yf_i = r_i \sin \phi_i \sin \tht_i, \;\; \zf_i = r_i \cos \phi_i, \;\; \forall i \in \I.
\end{equation}
Since $\mf{q}_j(t) +\mf{q}_k(t) =0$ and $m_j=m_k=1$, we have
$$ r_k(t) = r_j(t), \;\; \phi_k(t) = \pi - \phi_j(t), \;\; \tht_k(t) = \pi + \tht_j(t). $$

The following asymptotic properties are crucial in our proofs of the deformation lemmas.    
\begin{prop}
\label{prop:sundman} For any $i \in \I$ and $t>0$ small enough, 
$$ r_i(t) = C_1t^{\frac{2}{3}}+o(t^{\frac{2}{3}}), \;\; \dot{r}_i(t) = C_2 t^{-\frac{1}{3}} +o(t^{-\frac{1}{3}}).$$
\end{prop}
This is the well-known Sundman's estimates, for a proof see \cite{FT04}. 

\begin{prop} \label{prop:angle}
For any $i \in \I$, there exist $\phi^+_i \in [0, \pi]$ and $\tht^+_i \in \rr$ satisfying
\begin{enumerate}
\item[(a).] $\lim_{t \to 0^+} \phi_i(t) = \phi_i^+, \; \lim_{t \to 0^+} \tht_i(t) = \tht_i^+$,
\item[(b).] $\lim_{t \to 0^+} \dot{\phi}_i(t) = \lim_{t \to 0^+} \dot{\tht}_i(t) = 0$,
\item[(c).] $\phi^+_k = \pi - \phi^+_j, \; \tht_k^+ = \pi + \tht^+_j.$
\end{enumerate}
\end{prop}
This proposition implies $m_j$ and $m_k$ must approach to the binary collision along some definite directions. A detailed proof can be found in \cite{Y15b} and \cite{Y16s}.

Now we will state our first deformation lemma. It is a local property in nature and shows that after a local deformation of the path $q$ in a small neighborhood of the moment $t=0$, we can get ride of the isolated collision and obtain a new path with action value strictly smaller than $q$'s. Up to our knowledge, this type of local deformation first appeared in an unpublished paper by Montgomery \cite{Mo99} and further developed in the thesis of Venturelli \cite{Ve02}. For an isolated binary collision, a better result was obtained by the author in \cite{Y15b}. 
\begin{lm}
\label{lm:dfm1} For any $\om \in \rr$, if $\tht_j^+ \ne \frac{\pi}{2}( \text{mod } 2\pi)$ (resp. $\tht_j^+ \ne -\frac{\pi}{2}( \text{mod } 2\pi)$), then for positive $\ep$ and $\dl$ small enough, there is a $q^{\ep} \in H^1([0, T], \rr^{3N})$ $($a local deformation of $q$ near $t=0$ $)$ satisfying the following properties.

\begin{enumerate}
\item[(a).] If $i \in \N \setminus \I$, $q^{\ep}_i(t) = q_i(t)$, $\forall t \in [0, T]$, and if $i \in \I$,  
$$  \begin{cases}
q^{\ep}_i(t)= q_i(t), \; & \forall t \in [\dl, T]; \\
|q^{\ep}_i(t) - q_i(t)| \le \ep, \; & \forall t \in [0, \dl].
\end{cases}$$
\item[(b).] $x^{\ep}_i(0)= x_i(0)$ and $z^{\ep}_i(0)= z_i(0)$, $\forall i \in \I$. Furthermore $\qe_j(0)= \R_{xz} \qe_k(0)$ with $y^{\ep}_j(0) = -y^{\ep}_k(0)< 0$ $($resp. $ y^{\ep}_j(0) =- y^{\ep}_k(t) > 0).$ 
\item[(c).] $\ao(\qe; T) < \ao(q; T)$.
\end{enumerate}
\end{lm}
\begin{rem} \label{rem: dfm1}
The boundary condition $y_j^{\ep}(0) > 0$ (resp. $y_j^{\ep}(0)< 0$) listed in property (b) of the above lemma is directly related with the $\xi$-topological constraints introduced in the previous section. Because of this boundary condition, the above lemma does not hold, when $\tht_j^+ = \frac{\pi}{2}( \text{mod } 2\pi)$ (resp. $\tht_j^+ = -\frac{\pi}{2}( \text{mod } 2\pi)$). By the result of Gordon (\cite{Go77}), we know that a local deformation result like above does not exist in this case.
\end{rem}

\begin{proof}
First for $\om=0$, the above result is the same as Lemma 2.1 in \cite{Y16s}. A proof can be given based on Terracini's blow-up technique and a basic result of the Kepler problem, which says the zero energy collision-ejection solution which connects two different points with the same distance to the origin, has action value strictly large than the direct and indirect arcs joining these two points with the same transfer time (\cite{FGN11}). In the planar case a detailed proof can be found in \cite[Proposition 4.3]{Y15b}. The spatial case can be proven similarly. 

When $\om \ne 0$, we can reduce the problem to the case with $\om=0$. To see this, recall that for any path $\qt \in H^1([0, T], \rr^{3N})$, the corresponding path $\ej \qt \in H^1([0, T], \rr^{3N})$, see \eqref{eq: q om}, satisfies $\A(\ej \qt; T) = \ao(\qt; T)$. 
\end{proof}

To deal with the cases that are not covered by Lemma \ref{lm:dfm1} (see Remark \ref{rem: dfm1}), new techniques other than local deformations have to been introduced, so some global property of the collision solution can be used. This is why the monotone constraints were introduced in \cite{Y15c}. Here we reproduce the key results under weaker conditions following \cite{Y16s}.

\begin{dfn} \label{dfn:sepa}
Given two arbitrary subsets $\I_0, \I_1$ of $\N$ satisfying 
$$ \I_0 \cup \I_1 =\N \setminus \I, \;  \; \I_0 \cap \I_1 =\emptyset, \; \text{ and } \; \I_0 \cup \I_1 \neq \emptyset. $$ 
We say $q(t), t \in [0,T],$ is \textbf{$x$-separated} $($by $m_j$ and $m_k)$, if 
\begin{enumerate}
\item[(i).] $x_k(T) \le x_k(t) \le x_k(0)=x_j(0) \le x_j(t) \le x_j(T),$ $\forall t \in [0, T]$,
\item[(ii).] 
$x_i(t) \le x_k(T)$, if $i \in \I_0$ and $x_i(t) \ge x_j(T)$, if $i \in \I_1$, $\forall t \in [0, T]$,
\end{enumerate}
and \textbf{$z$-separated} $($by $m_j$ and $m_k )$, if
\begin{enumerate}
\item[(iii).] $z_k(T) \le z_k(t) \le z_k(0)=z_j(0) \le z_j(t) \le z_j(T),$ $\forall t \in [0, T]$,
\item[(iv).] 
$z_i(t) \le z_k(T)$, if $i \in \I_0$, and $z_i(t) \ge z_j(T)$, if $i \in \I_1$, $\forall t \in [0, T]$.
\end{enumerate}
\end{dfn}

\begin{lm}
 \label{lm:dfm2} For any $\om \in \rr$, when $\tht_j^+= \pm \frac{\pi}{2} (\text{mod } \pi)$, if $q(t), t \in [0, T]$ is $z$-separated and $z_j(T) > z_k(T)$, then for $\ep>0$ small enough, there is a new path $\qe: [0,T] \to \rr^{3N}$ defined by
 \begin{equation*}
  \begin{cases}
   \qe_i(t) &= q_i(t) - \ep^2 \e_3, \; \forall t \in [0,T], \text{ if } i \in \mathbf{I}_0,\\
   \qe_i(t) &= q_i(t) +\ep^2 \e_3, \; \forall t \in [0,T], \text{ if } i \in \mathbf{I}_1, \\
   \qe_j(t) &= q_j(t) +\ep^2 \e_3, \;\; \qe_k(t) = q_k(t) - \ep^2 \e_3, \; \forall t\in [\ep, T], \\
   \qe_j(t) &= q_j(t) + t(2\ep -t) \e_3, \;\; \qe_k(t) = q_k(t) -t(2\ep -t) \e_3, \; \forall t \in [0, \ep],
  \end{cases}
 \end{equation*}
 which satisfies $\ao(\qe; T) < \ao(q; T)$.

 Moreover when $\om=0$, if $q(t), t \in [0, T],$ is $x$-separated and $x_j(T) > x_k(T)$, then the above result still holds after replacing each $\e_3$ by $\e_1.$
\end{lm}

\begin{rem}
\label{rem: coercive center of mass} In the above lemma if we fix the center of mass of $q(t)$ at origin, for the new path $\qe(t)$, its center of mass may not be at the origin. This is why we do not make such assumption in this paper. 
\end{rem}
% We point out that the conditions $x_j(T) < x_k(T)$ and $z_j(T)< z_k(T)$ are crucial in the proof of Lemma \ref{lm:dfm2}.

\begin{proof}
When $\om=0$, the corresponding result have already been proven in \cite[Lemma 2.3]{Y16s}. When $\om \ne 0$, following the same argument given in the proof of Lemma \ref{lm:dfm1}, we may reduce the problem to the case $\om =0$, by considering the path $\ej q(t)$, $t \in [0, T]$ and the action functional $\A$. We point out that for $\om \ne 0$, this only works when $q(t)$ is $z$-separated, but not $x$-separated, as $\ej q(t)$ will still be $z$-separated, when $q(t)$ is, but may not be $x$-separated, when $q(t)$ is. 
\end{proof}

This finishes the first half of this section. In the second half, we do not assume $|\I|=2$ anymore, but only $|\I| \ge 2$. Furthermore we define 
\begin{equation}
\label{eq: T} \mf{T}:=\big\{ \tau = (\tau_i)_{i \in \N}| \; \tau_i \in \{ 0, \pm 1 \}  \big\}.
\end{equation} 

\begin{lm}
\label{lm:dfm4} If the collision solution $q(t), t \in [0, T]$, is contained in a plane parallel to the $xy$-plane, i.e., $z_i(t) \equiv \text{Constant}$, $\forall t \in [0, T]$ and $\forall i \in \N$), then for $\ep>0$ small enough and any $\tau \in \mf{T}$ satisfying
\begin{equation}
\label{eqn:cond tau} \tau_{i_0} \ne \tau_{i_1}, \text{ for some } \{i_0 \ne i_1\} \subset \I, \text{ and } \tau_i = 0, \; \forall i  \in \N \setminus \I,
\end{equation} 
there is an $f \in H^1([0, T], \rr)$ with 
\begin{enumerate}
\item[(a).] $f(t)=1$, $\forall t \in [0, \dl_1]$, for some $ \dl_1=\dl_1(\ep)>0$ small enough,
\item[(b).] $f(t) =0$, $\forall t \in [\dl_2, T]$, for some $ \dl_2=\dl_2(\ep) > \dl_1$ small enough,
\item[(c).] $f(t)$ is decreasing for $t \in [\dl_1, \dl_2]$, 
\end{enumerate}
such that for any $\om \in \rr$, the path $\qe = (\qe_i)_{i \in \N} \in H^1([0, T], \rr^{3N})$ defined by 
$$ \qe_i(t)=q_i(t)+ \ep f(t) \tau_i \e_3, \;\; \forall t \in [0, T], \; \forall i \in \N, $$ 
satisfies $\ao(\qe; T)< \ao(q; T)$. 
\end{lm}
\begin{rem}
When applying the above lemma to problems with symmetric constraints, the second part of the conditions in \eqref{eqn:cond tau}, i.e.
$$ \tau_i =0, \;\; \forall i \in \N \setminus \I, $$
need to be modified correspondingly, so that the deformed path will still satisfies the symmetric constraints. 
\end{rem}

\begin{proof}
By the same argument given in the Lemma \ref{lm:dfm1} and \ref{lm:dfm2}, the cases $\om \ne 0$ can be reduce to the case $\om=0$ by considering the path $\ej q(t)$ and the action functional $\A_0$. 

When $\om =0$, a detailed proof of this result with $q(t)$ belonging to a one dimensional subspace of $\rr^3$ can be found in \cite[Appendix]{Y15c}. The proof is exactly the same when it is contained in plane. As the proof works as long as the paths when deformed, are always along the directions that are orthogonal to the plane, where the path $q(t)$ belongs to. 
\end{proof}

All the results in this section are stated for a collision solution $q(t)$, $t \in [0, T]$ with at least one collision at the moment $t=0$. Meanwhile similar results can be obtained when we assume $q(t)$ has at least one collision at the moment $t=T$. We will not repeat the details here.

%a proof of this lemma is based on a local deformation property introduced by Montgomery and the blow-up technique introduced by Terracini \cite{FT04}, also see \cite{Ve01} and \cite{Y15b}. The main difference is that in the above references the configurations that can be served as the directions of local deformation that will low the action value need to be \emph{centered}. In the above notation, it simply means $\tau \in \mf{T}$ should satisfies $ \sum_{i \in \N} \tau_i=0$. In the above lemma, we need to make local deformations along $\tau$'s, which are not centered. Without this condition, the action valued of the deformed path may not be strictly smaller than the original one. However for Lemma \ref{lm:dfm4}, the path $q(t)$ is contained in a plane and we are only deforming along the directions that are orthogonal to this plane, the desired result will still hold. 

%---------------------------------------------------

\section{planar linear chains in the spatial N-body problem} \label{sec: fixed}

In this section only the non-rotating frame will be considered and the proofs of Theorem \ref{thm: linear chain} and \ref{thm: linear chain extra sym} will be given. First we prove a lemma, which shows a minimizer of $\A$ in $\ldn_{\xi}$ must be planar and satisfy certain monotone property along the $x$ and $z$-direction. In particular it shows the \emph{monotone constraints} required in \cite{Y15c} is  unnecessary, as it is a consequence of the minimization property of the path.

\begin{lm}
\label{lm: monotone fixed} For any $\xi \in \Xi_N$, if $q \in \ldn_{\xi}$ is a minimizer of $\A$ among all loops in $\ldn_{\xi}$ satisfying \eqref{eq: coercive x} and \eqref{eq: coercive z}, then it must satisfy the following properties: 
\begin{enumerate}
\item[(a).] either $x_0(t_1) \le x_0(t_2)$ or $x_0(t_1) \ge x_0(t_2)$ always holds, $\forall 0 \le t_1 \le t_2 \le \pi$;
\item[(b).] either $z_0(t_1) \le z_0(t_2)$ or $z_0(t_1) \ge z_0(t_2)$ always holds, $\forall 0 \le t_1 \le t_2 \le \pi$;
\item[(c).] For all $t \in \rr$, $q_0(t)$ belongs to a fixed plane, which is invariant under the action of $\R_{xz}$.
\end{enumerate}
\end{lm}

\begin{proof}
We give a detailed proof of property (b), while property (a) can be proven similarly. The main idea is that using $q$ we can build a new path in $\ldn_{\xi}$, such that its action value is strictly smaller than $q$'s, when property (b) does not hold, and the same as $q$'s, when property (b) holds.  

Let's define a new path $\qt(t)=(\qt_i(t))_{i \in \N} \in H^1([0, \pi/N], \rr^{3N}\}$, such that 
$$ \xt_i(t) = x_i(t) \text{ and } \yt_i(t) = y_i(t), \; \; \forall t \in  [0, \frac{\pi}{N}] \text{ and } i \in \N. $$
Meanwhile $\ztl_i(t)$ will be defined following an inductive procedure. First, let 
$$ \ztl_0(t) = \int_0^t |\zd_0(s)| \,ds + C_0, \;\; \forall t \in [0, \frac{\pi}{N}], $$
for some constant $C_0$ (to be determined later) and 
$$ \ztl_{N-1}(t)= -\int_0^t |\zd_{N-1}(s)| \,ds + C_{N-1}, \;\; \forall t \in [0, \frac{\pi}{N}], $$
with a proper constant $C_{N-1}$, such that $\ztl_0(\frac{\pi}{N}) = \ztl_{N-1}(\frac{\pi}{N}).$

Second assume $\ztl_j(t), \ztl_{N-1-j}(t)$ are defined for all $0 \le j \le i-1$, for some $1 \le i \le [\frac{N-1}{2}]$, then we set 
\begin{equation} \label{eq: ztl i N-1-i}
\forall t \in [0, \frac{\pi}{N}], \;\; \begin{cases}
\ztl_i(t) & = \int_0^t |\zd_i(s)| \,ds + C_i, \\
\ztl_{N-1-i}(t) & = -\int_0^t |\zd_{N-1-i}(s)| \,ds + C_{N-1-i},
\end{cases}
\end{equation}
with the proper $C_i$ and $C_{N-1 -i}$, such that
$$ \ztl_i(0) = \ztl_{N-i}(0), \;\; \ztl_i(\frac{\pi}{N}) =\ztl_{N-1-i}(\frac{\pi}{N}). $$
Notice that for an odd $N$ ($N=2n+1$), when $i = [\frac{N-1}{2}] =n$, $i=N-1-i =n$, so in this case we will just define $\ztl_i(t)$ as in \eqref{eq: ztl i N-1-i} with a $C_i$, such that $\ztl_i(0)= \ztl_{N-i}(0)$.

\begin{figure}
  \centering
  \includegraphics[scale=0.70]{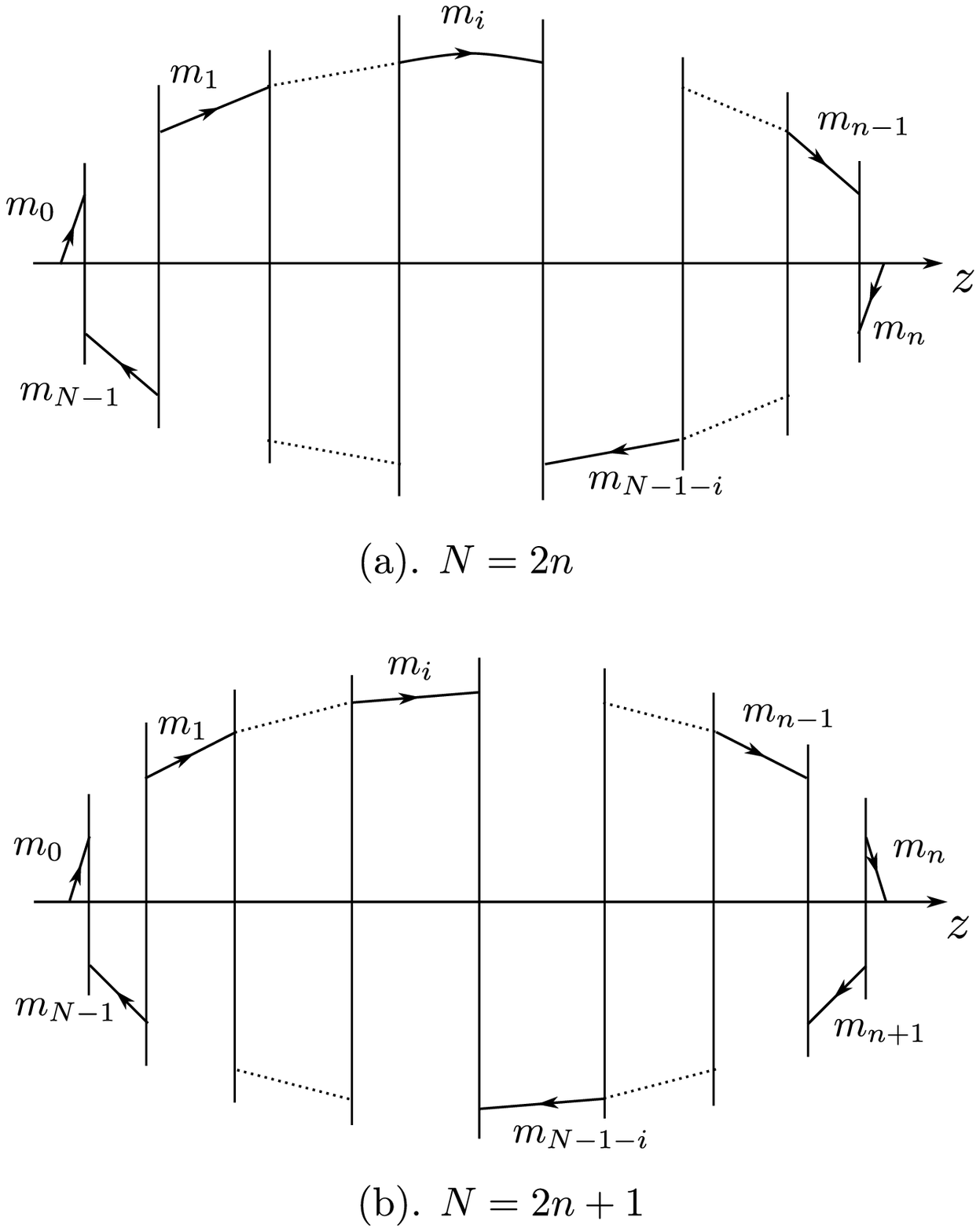}
  \caption{Projection of $\qt$ to the $yz$-plane }
  \label{fig:monotone}
\end{figure}

By the above definition, it is not hard to see $\qt \in \ldn_\xi$ and satisfies property (b). Meanwhile a proper value can always be found for $C_0$, such that \eqref{eq: coercive z} holds for $\qt$. 

With $\qt$ defined as above, the integration of the kinetic energy of each mass is the same as the corresponding one in $q$,
\begin{equation}
\label{eq: qt kinetic} \ey \int_0^{\pi/N} |\qd_i(t)|^2 \,dt = \ey \int_0^{\pi/N} |\dot{\qt}_i(t)|^2 \,dt, \;\; \forall i \in \N.
\end{equation}
Hence to prove property (b), it is enough to show the integration of the potential energy along the path $\qt$ is strictly smaller than along the path $q$, if the results in property (b) do not hold and the same, if they do. To prove this, let $\sg \in \mc{S}_{\N}$ be a permutation on the index set defined by
\begin{equation} \label{eq: sg}
\sg^{-1}(i) = \begin{cases}
2i, \; & \text{ if } 0 \le i \le [\frac{N-1}{2}]; \\
2(N-i)-1, \; & \text{ if } [\frac{N-1}{2}]+1 \le i \le N-1. 
\end{cases}
\end{equation}
Then for any $q^* \in H^1([0, \pi/N], \rr^{3N})$, its action value can be written as 
$$ \A(q^*; \frac{\pi}{N})= \sum_{k =0}^{N-1} \A^k(q^*; \frac{\pi}{N}),$$
where 
$$ \A^k(q^*; \frac{\pi}{N})=  \begin{cases} \ey \int_0^{\frac{\pi}{N}} |\dot{q}^*_{\sg(0)}(t)|^2 \,dt, \; & \text{ if } k=0, \\
 \int_0^{\frac{\pi}{N}} \ey |\dot{q}^*_{\sg(k)}|^2 + \sum_{i =0}^{k-1} \frac{1}{|q^*_{\sg(i)}-q^*_{\sg(k)}|} \,dt, \; & \text{ if } 1 \le k \le N-1.
\end{cases}
$$

By the definition of $\qt$, it is obvious $\A^0(\qt; \frac{\pi}{N})= \A^0(q; \frac{\pi}{N})$. We claim 
$$ \A^k(\qt; \frac{\pi}{N}) \le \A^k(q; \frac{\pi}{N}), \;\; \forall k \in \N \setminus \{0 \},$$
 and the above inequalities are equalities if and only if one of the following holds: 
\begin{equation}
\label{eq: increasing} \forall 0 \le i \le k, \; \begin{cases}
\zd_{\sg(i)}(t) \ge 0, \; \forall a.e. \; t \in [0, \frac{\pi}{N}], \; & \text{ if } i \text{ is } even; \\
\zd_{\sg(i)}(t) \le 0, \; \forall a.e. \; t \in [0, \frac{\pi}{N}], \; & \text{ if } i \text{ is } odd, 
\end{cases}
\end{equation}

\begin{equation}
\label{eq: decreasing} \forall 0 \le i \le k, \; \begin{cases}
\zd_{\sg(i)}(t) \le 0, \; \forall a.e. \; t \in [0, \frac{\pi}{N}], \; & \text{ if } i \text{ is } even; \\
\zd_{\sg(i)}(t) \ge 0, \; \forall a.e. \; t \in [0, \frac{\pi}{N}], \; & \text{ if } i \text{ is } odd, 
\end{cases}
\end{equation}
%Due to the symmetric constraint \eqref{eq: increasing} and \eqref{eq: decreasing} correspondingly are equivalent to 
%$$ \zd_0(t) \ge 0 \text{ or } \zd_0(t) \le 0, \; \forall t \in [0, \pi/N].$$

First let's prove the claim for $k=1$. Recall that by \eqref{eq: sg}, $\sg(0)=0$ and $\sg(1)=N-1$. By the definition of $\qt(t)$, $ \forall t \in [0, \pi/N]$, 
\begin{align*}
|x_{\sg(0)}(t)- x_{\sg(1)}(t)| & = |\xt_{\sg(0)}(t) - \xt_{\sg(1)}(t)|, \\
|y_{\sg(0)}(t)- y_{\sg(1)}(t)| & = |\yt_{\sg(0)}(t)-\yt_{\sg(1)}(t)|,
\end{align*}
Meanwhile $z_{\sg(0)}(\frac{\pi}{N})=z_{\sg(1)}(\frac{\pi}{N})$ and $\ztl_{\sg(0)}(\frac{\pi}{N})=\ztl_{\sg(1)}(\frac{\pi}{N})$ imply
\begin{equation}
\label{eq: diff potential z ztl} \begin{split}
|z_{\sg(0)}(t) - z_{\sg(1)}(t)| & = |z_{\sg(0)}(t)- z_{\sg(0)}(\pi/N) + z_{\sg(1)}(\pi/N) - z_{\sg(1)}(t) | \\
                                & \le |z_{\sg(0)}(t) - z_{\sg(0)}(\pi/N)| + |z_{\sg(1)}(\pi/N) - z_{\sg(1)}(t)| \\ 
                                & \le \int_t^{\pi/N} |\zd_{\sg(0)}(s)| \,ds + \int_t^{\pi/N} |\zd_{\sg(1)}(s)| \,ds \\
                                &= \ztl_{\sg(0)}(\pi/N)-\ztl_{\sg(0)}(t)+ \ztl_{\sg(1)}(t)-\ztl_{\sg(1)}(\pi/N) \\
                                & = \ztl_{\sg(1)}(t) - \ztl_{\sg(0)}(t).
\end{split}
\end{equation}
In particular the inequalities in \eqref{eq: diff potential z ztl} are equalities for any $t \in [0, \frac{\pi}{N}]$, if and only if \eqref{eq: increasing} or \eqref{eq: decreasing} holds. Together with \eqref{eq: qt kinetic}, it proves our claim for $k=1$.

By induction, assume our claim holds for all $1 \le i \le k-1$ for some $1 \le k \le N-1$, we will show it must hold for $k$ as well. The details for $k$ being even will be given below, the proof for $k$ being odd is similar and will be omitted. Again by the definition of $\qt$, for any $t \in [0, \pi/N]$ and $1 \le i \le k-1$,
\begin{align*}
|x_{\sg(i)}(t)- x_{\sg(k)}(t)| & = |\xt_{\sg(i)}(t) - \xt_{\sg(k)}(t)|, \\
|y_{\sg(i)}(t)- y_{\sg(k)}(t)| & = |\yt_{\sg(i)}(t)-\yt_{\sg(k)}(t)|. 
\end{align*}
Meanwhile using \eqref{eq: symmetric boundary moment 1}, when $i$ is even, we have 
\begin{align*}
|&z_{\sg(i)}(t)- z_{\sg(k)}(t)| \\
&= | z_{\sg(i)}(t) - \sum_{j=i}^{k-2} \big( z_{\sg(j)}(\frac{\pi}{N}) - z_{\sg(j+1)}(\frac{\pi}{N}) +z_{\sg(j+1)}(0) - z_{\sg(j+2)}(0) \big) -z_{\sg(k)}(t)| \\
& \le |z_{\sg(i)}(t)-z_{\sg(i)}(\frac{\pi}{N})| + \sum_{j=i+1}^{k-1}|z_{\sg(j)}(\frac{\pi}{N})-z_{\sg(j)}(0)| + |z_{\sg(k)}(0)-z_{\sg(k)}(t)| \\
& \le \int_t^{\pi/N} |\xd_{\sg(i)}(s)| \,ds + \sum_{j=i+1}^{k-1} \int_0^{\pi/N} |\zd_{\sg(j)}(s)| \,ds + \int_0^t |\zd_{\sg(k)}(s)| \,ds \\
& = |\ztl_{\sg(i)}(t)-\ztl_{\sg(i)}(\frac{\pi}{N})| + \sum_{j=i+1}^{k-1}|\ztl_{\sg(j)}(\frac{\pi}{N})-\ztl_{\sg(j)}(0)| + |\ztl_{\sg(k)}(0)-\ztl_{\sg(k)}(t)| \\
& = |\ztl_{\sg(i)}(t) - \ztl_{\sg(k)}(t)|,
\end{align*}
and when $i$ is odd, we have 
\begin{align*}
|&z_{\sg(i)}(t)- z_{\sg(k)}(t)| \\
&= | z_{\sg(i)}(t) - \sum_{j=i}^{k-2} \big( z_{\sg(j)}(0) - z_{\sg(j+1)}(0) +z_{\sg(j+1)}(\frac{\pi}{N}) - z_{\sg(j+2)}(\frac{\pi}{N}) \big) -z_{\sg(k)}(t)| \\
& \le |z_{\sg(i)}(t)-z_{\sg(i)}(0)| + \sum_{j=i+1}^{k-1}|z_{\sg(j)}(0)-z_{\sg(j)}(\frac{\pi}{N})| + |z_{\sg(k)}(0)-z_{\sg(k)}(t)| \\
& \le \int_0^t |\xd_{\sg(i)}(s)| \,ds + \sum_{j=i+1}^{k-1} \int_0^{\pi/N} |\zd_{\sg(j)}(s)| \,ds + \int_0^t |\zd_{\sg(k)}(s)| \,ds \\
& = |\ztl_{\sg(i)}(t)-\ztl_{\sg(i)}(0)| + \sum_{j=i+1}^{k-1}|\ztl_{\sg(j)}(\frac{\pi}{N})-\ztl_{\sg(j)}(0)| + |\ztl_{\sg(k)}(0)-\ztl_{\sg(k)}(t)| \\
& = |\ztl_{\sg(i)}(t) - \ztl_{\sg(k)}(t)|.
\end{align*}
Again the above inequalities are equalities for all $t \in [0, \frac{\pi}{N}]$, if and only if \eqref{eq: increasing} or \eqref{eq: decreasing} holds. Together with \eqref{eq: qt kinetic}, it implies our claim for $k$. This finishes our proof of property (b). 

For property (c), notice that if we rotate a path from $\ldn_{\xi}$ around the $y$-axis by an arbitrary angle, it will still belong to $\ldn_{\xi}$ and its action value will not be changed. Recall that the $D_N$-symmetry implies $q_0(0)$ and $q_0(\pi)$ belongs to the $xz$-plane, so after rotate the entire path by a proper angle around the $y$-axis, we can make $z_0(0)= z_0(\pi)$. Then property (b) implies $z_0(t)$ is a constant for $t \in [0, \pi]$, which means the entire path is contained inside a plane parallel to the $xy$-plane and property (c) follows immediately. 
\end{proof}

%With the above two lemmas, Theorem \ref{thm: linear chain} basically reduced to the planar case that we have studied in \cite{Y15c}. However **********
Property (a) and (b) in the above lemma indicate the action minimizer satisfies some monotone property along the $x$ and $z$-axis. Under some extra conditions, we will show the monotonicity is in fact strict. 
\begin{lm}
\label{lm: strict monotone fixed} For any $\xi \in \Xi_N$, if $q \in \ldn_\xi$ is a minimizer of $\ao$ among all loops in $\ldn_{\xi}$ satisfying \eqref{eq: coercive x} and \eqref{eq: coercive z}, then the following results hold. 
\begin{enumerate}
\item[(a).] If $x_0(t)$ is not a constant for all $t \in \rr$, then $q$ is a collision-free $2\pi$-periodic solution of \eqref{eq:nbody}. Moreover $\xd_0(t)=0$, if and only if $t \in \{0, \pi\}$ and 
$$  \begin{cases}
\xd_0(t) >0 \; (\text{resp.} <0), & \; \text{ if } t \in (0, \pi), \\
\xd_0(t) <0 \; (\text{resp.} >0), & \; \text{ if } t \in (\pi, 2\pi). 
\end{cases} $$
\item[(b).] If $z_0(t)$ is not a constant for all $t \in \rr$, then $q$ is a collision-free $2\pi$-periodic solution of \eqref{eq:nbody}. Moreover $\zd_0(t)=0$, if and only if $t \in \{0, \pi\}$ and 
$$  \begin{cases}
\zd_0(t) >0 \; (\text{resp.} <0), & \; \text{ if } t \in (0, \pi), \\
\zd_0(t) <0 \; (\text{resp.} >0), & \; \text{ if } t \in (\pi, 2\pi). 
\end{cases} $$
\end{enumerate}
\end{lm}
We postpone the proof of the above lemma for a moment, as in Section \ref{sec: rotate} a more general result (Lemma \ref{lm: strict monotone rotate}) will be proven, which includes the above lemma as a special case. As explained in Remark \ref{rem: strict monotone}, property (a) and (b) in the above lemma imply certain strict monotone property of an action minimizer along the $x$ and $z$-direction correspondingly. 
 
Now we are ready to prove Theorem \ref{thm: linear chain} and \ref{thm: linear chain extra sym}. 
\begin{proof} [Proof of Theorem \ref{thm: linear chain}] First we need to show the existence of at least one action minimizer. Choose an arbitrary sequence of loops $\{q^n\}$ from $\ldn_{\xi}$, due to the $D_N$-symmetry, each loop must satisfy \eqref{eq: coercive y}. If \eqref{eq: coercive x} and \eqref{eq: coercive z} hold for each loop as well, then by Poincar\'e's inequality 
\begin{equation}
\int_{0}^{2\pi} |\qd^n_i(t)|^2 \,dt \ge \|q^n_i\|^2_{\hnm}, \;\; \forall i \in \N. 
\end{equation}
As the potential $U$ is never negative, we get $\A(q^n; 2\pi)$ goes to infinity, when $\|q^n\|_{\hnm}$ goes to infinity. Then the existence of a minimizer follows from the lower semi-continuity of $\A$ and a standard argument in calculus of variation. 

Now let $q \in \ldn_{\xi}$ be an action minimizer of $\A$ in $\ldn_{\xi}$. We claim the following can not happen
$$ x_i(t) \equiv \text{Constant}, \; z_i(t) \equiv \text{Constant}, \; \forall t \in [0, \frac{\pi}{N}] \text{ and } \forall i \in \N. $$
Because otherwise all the masses will move inside a straight line parallel to the $y$-axis all the time. Then the symmetric and topological constraints will implies the existence of at least one isolated collision of $q(t)$ for some moment $t$, and by Lemma \ref{lm:dfm4}, we can find another path from $\ldn_{\xi}$ whose action value is strictly smaller than $q$'s, which is absurd. A detailed proof of this can be found in \cite[Appendix]{Y15c}. Let's point out that in the above paper, certain monotone constraints were imposed a priori, and during the proof a lot of effort were made to ensure the monotone constraints were satisfied for the path we found. In current setting, as no monotone constraints were imposed in advance, all these will be unnecessary and the proof can be simplified significantly. We left the details to the reader. 

By the above claim, either $x_0(t)$ or $z_0(t)$ will not be a constant for all $t \in \rr$. Then our result follows directly from Lemma \ref{lm: monotone fixed} and \ref{lm: strict monotone fixed}.
\end{proof}

%For the corresponding minimizer $q$, if neither $x_0(t) \equiv \text{Constant}$ nor $z_0(t) \equiv \text{Constant}$ is true, for all $t \in \rr$, a few words may be needed to justify that $q$ is contained in a plane invariant under the action of $\R_{xz}$. Let's assume a minimizer $q$ is not contained inside such a plane, then we may rotate $q$ by a proper angle around the $y$-axis to make $z_0(0)= z_0(\pi)$. In this case the new loop will still belong to $\ldn_{\xi}$ and have the same action value. Under the assumption the new loop can not lie inside a plane with a constant $z$-coordinate, which then violates property (b) in \ref{lm: monotone fixed}. This finishes our entire proof. 

\begin{proof}[Proof of Theorem \ref{thm: linear chain extra sym}] 
Recall that the $H_N$-symmetry implies each loop from $\Lmd^{H_N}$ must satisfies \eqref{eq: coercive y} and \eqref{eq: coercive z}. Then the first part of the theorem can be proven similarly as Theorem \ref{thm: linear chain}, except some extra care has to been taken when we use Lemma \ref{lm: monotone fixed}. Notice that in the proof of this lemma, a new path $\qt$ was constructed, while such a path always belongs to $\ldn_{\xi}$, it may not belong to $\Lmd^{H_N}_{\xi}$. What's needs to  be done is that after getting the path $\qt$ following the process given in the proof of Lemma \ref{lm: monotone fixed}, we need to rotate it around the $y$-axis by a proper angle and make a linear translation of it along the $z$-axis by a proper constant, so that conditions in \eqref{eq: q0 HN symmetry} will be satisfied. Then the path shall belong to $\Lmd^{H_N}_{\xi}$. 

For the second part, let's assume $x_0(t)$ is not a constant for all $t \in \rr$, then by Lemma \ref{lm: strict monotone fixed}, $\xd_0(t)$ is either always positive or always negative, for all $t \in (0, \pi)$. This is a contradiction to the fact $ x_0(0)= x_0(\pi)$, which is required by the $H_N$-symmetry, see \eqref{eq: q0 HN symmetry}. This means $x_0(t) = \text{Constant}$, for all $t \in \rr$. Meanwhile since $q$ satisfies condition \eqref{eq: coercive x}, such a constant must be zero. 
\end{proof}

%---------------------------------------------

\section{Coercivity in rotating frame} \label{sec: coercive}

From now on we will consider our problem in a coordinate frame rotating around the $z$-axis with a constant angular velocity $\om \in \rr$ as explained in Section \ref{sec:intro}. In this section, we will study the coercivity of the action functional $\ao$ in $\ldn_{\xi}$ for different values of $\om$ and $\xi \in \Xi_N$.

\begin{lm}
\label{lm: coercive non integer om} Given an arbitrary sequence $\qn \in \Lmd$ satisfying 
\begin{equation}
\label{eq: coercive zn Lmd} [z_i]:= \frac{1}{2\pi} \int_0^{2\pi} z_i(t) \,dt = 0, \;\; \forall i \in \N,
\end{equation}
for any $\om \in \rr \setminus \zz$, if $\|\qn\|_{\hnm}$ goes to infinity, then $\ao(\qn; 2\pi)$ goes to infinity,.
\end{lm}

\begin{proof}

Given a loop $q =(\zt_i, z_i)_{i \in \N} \in \Lmd$, written in Fourier series, for each $i \in \N$, 
$$ \zt_i(t) = \sum_{k \in \zz} \hat{\zt}_{i, k} e^{\J kt}, \; \text{ where } \; \hat{\zt}_{i, k}= \frac{1}{2\pi} \int_0^{2\pi} \zt_{i}(t)e^{-\J kt} \,dt; $$
$$ z_i(t) = \sum_{k \in \zz} \hat{z}_{i, k} e^{\J kt}, \; \text{ where }\; \hat{z}_{i, k}= \frac{1}{2\pi} \int_0^{2\pi} z_{i}(t)e^{-\J kt} \,dt.  $$
Then 
$$ \| \zt_i\|^2_{\lnm} = \int_0^{2\pi} |\zt_i(t)|^2 \,dt = 2\pi \sum_{k \in \zz} |\hat{\zt}_{i, k}|^2,$$
$$  \| z_i\|^2_{\lnm} = \int_0^{2\pi} |z_i(t)|^2 \,dt =  2 \pi \sum_{k \in \zz} |\hat{z}_{i,k}|^2, $$
$$ \| \dot{\zt}_i\|^2_{\lnm} = \int_0^{2\pi} |\dot{\zt}_i(t)|^2 \,dt = 2\pi \sum_{k \in \zz^*} k^2 |\hat{\zt}_{i, k}|^2,$$
where $\zz^*= \zz \setminus \{0 \}$. As a result, 
\begin{equation}
\| \zt_i\|^2_{\hnm} = \| \zt_i\|^2_{\lnm} + \| \dot{\zt}_i\|^2_{\lnm}= 2\pi \sum_{k \in \zz} (k^2 +1) |\hat{\zt}_{i,k}|^2. 
\end{equation}
Meanwhile $\dot{\zt}_i(t)+ \J \om \zt_i(t) = \sum_{k \in \zz} \J (\om +k) \hat{\zt}_{i, k} e^{\J kt}$. Then 
\begin{equation} \label{eq: K om L2 norm}
\|\dot{\zt}_i+\J \om \zt_i\|^2_{\lnm} = 2\pi \sum_{k \in \zz} (\om + k)^2 |\hat{\zt}_{i, k}|^2. 
\end{equation}
Since $\om \notin \zz$, $\min_{k \in \zz} (\om+k)^2 \ge C_1 >0$, where $C_1$ only depends on $\om$. Hence  
\begin{equation} \label{eq: K om zt ge zt} 
\|\dot{\zt}_i+\J \om \zt_i\|^2_{\lnm} \ge 2\pi C_1  \sum_{k \in \zz} |\hat{\zt}_{i, k}|^2 = C_1 \| \zt_i\|^2_{\lnm}. 
\end{equation}

To get a similar estimate for $\dot{\zt}_i$, notice that for each $k \in \zz$, 
\begin{equation*}
\begin{split}
k^2 & \le (|k+\om| + |\om|)^2 = (k+\om)^2 + 2 |\om||k + \om| + \om^2 \\
    & \le (k+\om)^2 + 2|\om|(|k+\om|^2 +1) + \om^2 \le C_2(k+\om)^2 + C_3,
\end{split}
\end{equation*}
where the positive constants $C_2, C_3$ only depend on $\om$. Therefore 
\begin{equation} \label{eq: K om zt ge dot zt}
\begin{split}
\|\dot{\zt}_i\|^2_{\lnm} & = 2\pi \sum_{i \in \zz^*} k^2 |\hat{\zt}_{i, k}|^2 \le 2\pi C_2 \sum_{k \in \zz} (k+\om)^2 |\hat{\zt}_{i, k}|^2 + 2\pi C_3 \sum_{k \in \zz} |\hat{\zt}_{i, k}|^2 \\
& \le C_4 \|\dot{\zt}_i + \J \om \zt_i\|^2_{\lnm},
\end{split}
\end{equation} 
where the last inequality follows from \eqref{eq: K om L2 norm} and \eqref{eq: K om zt ge zt}. Together \eqref{eq: K om zt ge zt} and \eqref{eq: K om zt ge dot zt} imply 
\begin{equation}
\label{eq: K om zt ge H1 norm zt} \|\dot{\zt}_i+\J \om \zt_i\|^2_{\lnm} \ge C_5 \| \zt_i \|_{\hnm}^2. 
\end{equation}

For each $z_i(t)$, $i \in \N$, notice that condition \eqref{eq: coercive zn Lmd} implies $\hat{z}_{i, 0}=0$, then 
$$ \|\zd_i\|^2_{\lnm} = 2\pi \sum_{k \in \zz^*} k^2 |\hat{z}_i|^2 \ge 2\pi \sum_{k \in \zz^*} |\hat{z}_i|^2 = \| z_i\|^2_{\lnm}, $$
which means 
\begin{equation}
 \label{eq: z dot L2 ge z H1}\|\zd_i\|_{\lnm} \ge \ey \| z_i \|^2_{\hnm}. 
 \end{equation} 
Since $U(q)$ is always positive, \eqref{eq: K om zt ge H1 norm zt} and \eqref{eq: z dot L2 ge z H1} imply
$$ \ao(q; 2\pi) \ge \ey \int_0^{2\pi} \sum_{i \in \N} (|\dot{\zt}_i + \J \om \zt_i|^2 + |\dot{z}_i|^2) \,dt \ge C\|q\|^2_{\hnm},$$
for some constant $C$ only depending on $\om$. This finishes our proof.
\end{proof}

With the above lemma, we can give a proof of Theorem \ref{thm: coercive rotating non-integer}
\begin{proof} 

[\textbf{Theorem} \ref{thm: coercive rotating non-integer}] (a). If $\om \in [0, N] \setminus \zz$, by Lemma \ref{lm: coercive non integer om},  $\ao$ is coercive among all loops in $\ldn_{\xi}$ satisfying \eqref{eq: coercive z}, then the desired property follows directly from the lower semi-continuity of $\ao$ and a standard argument in calculus of variation. 

(b). If $\om =0$, the frame is fixed and the result has already been proven in Theorem \ref{thm: linear chain}. If $\om = N$, then at the moment $t=\pi/N$ the $x$ and $y$-axis come back to their original positions at the moment $t=0$, but with reversed directions. Since $[0, \pi/N]$ is a fundamental domain of the $D_N$-equivalent loops in $\ldn$, $q \in \ldn_{\xi}$ if and only if $e^{\J Nt}q \in \ldn_{\xi^*}$ with $\xi^*$ defined in \eqref{eq: xi star}. As a result, $q$ is a minimizer of $\A_{N}$ in $\ldn_{\xi}$ if and only if $e^{\J N t}q$ is a minimizer of $\A$ in $\ldn_{\xi^*}$. As the later case is a minimization problem in the non-rotating frame, the desired result again follows from Theorem \ref{thm: linear chain}. 
\end{proof}

To deal with the cases that $ \om \in (0, N) \cap \zz$, we need another lemma. Let $\zz_{N}=\langle g | \; g^N =1 \rangle$ be the cyclic group of order $N$ with the actions of $g$ defined as in \eqref{eq: DN}. Then each $q \in \lzn$ is a \textbf{simple choreographic loop}, i.e.
\begin{equation}
\label{eq: simple choreography path} q_i(t)= q_0(t + i \frac{2\pi}{N}), \;\; \forall t \in \rr, \; \forall i \in \N. 
\end{equation}

\begin{lm}
\label{lm: no coercive} For any $k \in [1, N-1] \cap \zz$, let  $\qn =(\zt^n_i, z^n_i)_{i \in \N} \in \lzn$ be a sequence of loops satisfying condition \eqref{eq: coercive z}. If $\|\qn\|_{\hnm} \to \infty$ and $\A_k(\qn; 2\pi)$ remain bounded, when $n \to \infty$, then the following properties hold. 
\begin{enumerate}
\item[(a).] $\| z_i^n\|_{\hnm}$ has a finite upper bound independent of $i$ and $n$. 
\item[(b).] $|\zt_i(t)| \to \infty$ uniformly on $t \in \rr$, for each $i \in \N$.
\item[(c).] After passing $q^n$ to a subsequence, for each $i \in \N$, there exists a $u_i \in \cc$ with $|u_i|=1$, such that $\frac{\ztn_i(t)}{|\ztn_i(t)|}$ converges uniformly to $e^{-\J k t}u_i$ on $t \in \rr$ correspondingly. \end{enumerate}

\end{lm}

\begin{proof}
We set $e^{\J kt}q^n =(\eta^n_i, z^n_i)_{i \in \N}=(e^{\J kt} \zt^n_i, z^n_i)_{i \in \N}$. Notice that $e^{\J kt}q^n$ is $2\pi$-periodic, as $k \in \zz$. 

(a). If $q \in \lzn$ satisfies condition \eqref{eq: coercive z}, by \eqref{eq: simple choreography path}, it must satisfies condition \eqref{eq: coercive zn Lmd} as well. Meanwhile the proof of Lemma \ref{lm: coercive non integer om}, \eqref{eq: z dot L2 ge z H1} holds here as well. Then property (a) follows directly from the inequality
$$ \A_k(q; 2\pi) \ge \ey \int_0^{2\pi} \sum_{i \in \N} |\zd_i|^2 \,dt. $$ 

(b). By \eqref{eq: simple choreography path}, it is enough to prove the result for $i =0$. Since $|\eta_0^n(t)|=|\zt_0^n(t)|$ and is $2\pi$-periodic, it is enough to show
\begin{equation} \label{eq; etn min unbound}
\min_{t \in [0, 2\pi]} |\etn_0(t)| \to \infty, \text{ as } n \to \infty. 
\end{equation}
First we claim the following weaker result holds
\begin{equation} \label{eq; etn max unbound}
\max_{t \in [0, 2\pi]} |\etn_0(t)| \to \infty, \text{ as } n \to \infty. 
\end{equation}
Assume \eqref{eq; etn max unbound} does not hold, then
\begin{equation}
\label{eq: etn 0 bound} \max_{ t \in [0, 2\pi]} |\etn_i(t)| =  \max_{ t \in [0, 2\pi]}  |\ztn_i(t)| \le C_2, \; \; \forall i \in \N. 
\end{equation}
Combining this with property (a) and the fact that $\| \qn \|_{\hnm} \to \infty$, it implies $\| \dot{\ztn_0}\|_{\lnm} \to \infty$. 

Meanwhile by the definition of $\etn_0(t)$, $\dot{\etn_0}(t) = e^{\J k t}(\dot{\ztn_0}(t) + \J k \ztn_0(t))$, so
\begin{equation}
\label{eq: etan dot ge} |\dot{\etn_0}(t)| \ge |\dot{\ztn_0}(t)| - k |\ztn_0(t)| \ge |\dot{\ztn_0}(t)|-k C_2. 
\end{equation}
where the last inequality follows from \eqref{eq: etn 0 bound}. Then by Cauchy-Schwartz inequality,
\begin{equation} \begin{split}
\int_0^{2\pi} |\dot{\etn_0}(t)|^2 \,dt & \ge \int_0^{2\pi} |\dot{\ztn_0}(t)|^2 \,dt -2k C_2 \int_0^{2\pi} |\dot{\ztn_0}(t)| \,dt - 2\pi k^2 C_2^2 \\
 &\ge \int_0^{2\pi} |\dot{\ztn_0}|^2 \,dt - C_3 \left(\int_0^{2\pi} |\dot{\ztn_0}|^2 \,dt \right)^{\ey} - C_4.
\end{split}
\end{equation}
where $C_3, C_4$ are positive constants independent of $n$. Since $\| \dot{\ztn_0}\|_{\lnm} \to \infty$, it implies $\|\dot{\etn_0}\|_{\lnm} \to \infty$. Then
\begin{equation}
 \label{eq; A om unbound} \A_k(\qn; 2\pi) = \A(e^{\J kt} \qn; 2\pi) \ge \ey \int_0^{2\pi} \sum_{i \in \N} |\dot{\etn_i}|^2 \,dt = \frac{N}{2} \int_0^{2\pi} |\dot{\etn_0}|^2 \,dt \to \infty.
 \end{equation} 
This contradicts the condition that $\A_{k}(\qn; 2\pi)$ remains bounded and proves \eqref{eq; etn max unbound}.

Meanwhile by Cauchy-Schwartz inequality
\begin{equation} \label{eq: eta dot max min}
\int_0^{2\pi} |\dot{\etn_0}|^2 \,dt \ge \frac{1}{2\pi} \left( \int_0^{2\pi} |\dot{\etn_0}| \,dt \right)^2 \ge \frac{1}{2\pi} (\max_{t \in [0, 2\pi]} |\etn_0(t)| - \min_{t \in [0, 2\pi]}|\etn_0(t)|)^2. 
\end{equation}
Assume \eqref{eq; etn min unbound} does not hold, then \eqref{eq; etn max unbound} and \eqref{eq: eta dot max min} again imply \eqref{eq; A om unbound}, which is a contradiction as we just explained. This finishes our proof of property (b). 

(c). Like property (b), we will just give the proof for $i =0$. Since $\etn_0(t)$ is $2\pi$-periodic, after passing to a subsequence, $\frac{\etn_0(t)}{|\etn_0(t)|}$ converges uniformly to a periodic function $v: \rr/2\pi \zz \to \cc$ satisfying $|v(t)|=1$, for any $t$. 

We claim $v(t) \equiv v(0)$, for all $t$. Otherwise there is a $t_0 \in (0, 2\pi)$, such that $v(t_0) \ne v(0)$. Together with the result just proved in property (b), it implies $|\etn_0(t_0)-\etn_0(0)| \to \infty$. Then by Cauchy-Schwartz inequality 
$$ \int_0^{2\pi} |\dot{\etn_0}|^2 \,dt \ge \frac{1}{2\pi} \left( \int_0^{2\pi} |\dot{\etn_0}| \,dt \right)^2 \ge \frac{1}{2\pi}|\etn_0(t_0)-\etn_0(0)|^2 \to \infty. $$ 
The rest of the claim follows from similar arguments as in property (b). 

Let $u_0=v(0)$, then $\frac{\etn_0(t)}{|\etn_0(t)|}$ converges uniformly to $u_0$ on $\rr$. Meanwhile as $\etn_0(t) = e^{\J kt}\ztn_0(t)$, we have $\frac{\ztn_0(t)}{|\ztn_0(t)|}$ converges uniformly to $e^{-\J kt} u_0$ on $\rr$.

\end{proof}

Now we finish this section with a proof of Theorem \ref{thm: coercive rotating integer}. 
\begin{proof}

[\textbf{Theorem} \ref{thm: coercive rotating integer}] (a). It will be enough to prove the coercive condition holds. By a contradiction argument, let's this does not hold, then there is a sequence $\qn=(\zt^n_i, z^n_i)_{i \in \N} \in \ldn_{\xi}$ satisfying \eqref{eq: coercive z}, such that $ \|\qn \|_{\hnm} \to \infty$ and $\A_k(\qn; 2\pi) < C $, for some finite constant $C$ independent of $n$. By Lemma \ref{lm: no coercive}, after passing $\qn$ to a subsequence, for each $i \in \N$, there exists a constant $u_i \in \cc$ with $|u_i|=1$ and a function $z_i \in H^1(\rr /2\pi \zz, \rr)$, such that $\frac{\ztn_i(t)}{|\ztn_i(t)|}$ converges uniformly to $e^{-\J k t} u_i$ and $z^n_i(t)$ converges uniformly to $z_i(t)$ on $t \in \rr$. 

Since $q^n(t)$ converges uniformly to $(e^{-\J k t}u_i, z_i(t))_{i \in \N}$ and $\qn \in \ldn_{\xi}$ for all $n$, the limiting loop $(e^{-\J k t}u_i, z_i(t))_{i \in \N}$ must belong to $\ldn_{\xi}$ as well. By the definition of the $D_N$-symmetry, $(e^{-\J kt}u_i, 0)_{i \in \N}$, which is the projection of the limiting loop, is contained in $\Lmd^{D_N}$. Meanwhile since the $\xi$-topological constraints are only imposed on the $y$-component, it is easy to see $(e^{-\J kt}u_i, 0)_{i \in \N}$ satisfies the $\xi$-topological constraints as well, so it belongs to $\ldn_{\xi}$. 

Notice that for $(e^{-\J kt}u_i, 0)_{i \in \N}$ to satisfy the $\zz_N$-symmetry, $(u_i)_{i \in \N}$ must form a regular $N$-gon in the $xy$-plane with the center of mass at the origin. Then to further satisfy the $D_N$-symmetry, the location of the masses must coincide with $\fn_n^+$ or $\fn_k^-$. This then implies $\xi = \xi(\fn_k^+)$ or $\xi(\fn_k^-)$, which is a contradiction.

(b). Let's assume $\xi=\xi(\fn^+_k)$ (the proof for $\xi= \xi(\fn_k^-)$ is exactly the same), and $\tilde{q}^n(t)= (\tilde{\zt}^n(t), 0)$ is a sequence of $2\pi$-periodic loops
$$ \tilde{\zt}^n(t)= \lmd_n e^{-\J k t} \fn^+_{k}, \; \forall t \in \rr,$$
with $\lmd_n >0$, for each $n$, and $\lim_{n \to \infty} \lmd_n=0$. By a straight forward computation, 
$$ \lim_{n \to \infty} \A_k(\tilde{q}^n; 2\pi)=0.$$
As we explained in Section \ref{sec:intro}, $\qn \in \Lmd_{\xi}^{D_N}$ and obviously it satisfies \eqref{eq: coercive z}. Therefore
$$ \inf\{ \A_k(q)| \; q \in \ldn_{\xi} \text{ satisfying } \eqref{eq: coercive z} \} \le \lim_{n \to \infty} \A_k(\tilde{q}^n; 2\pi) =0.$$
Meanwhile it is obvious $\inf\{ \A_k(q)| \; q \in \ldn_{\xi} \text{ satisfying } \eqref{eq: coercive z} \} \ge 0$. Hence
$$ \inf\{ \A_k(q)| \; q \in \ldn_{\xi} \text{ satisfying } \eqref{eq: coercive z} \} = \lim_{n \to \infty} \A_k(\tilde{q}^n; 2\pi) =0.$$
For the rest of the property, let $\qn \in \ldn_{\xi}$ be an arbitrary sequence satisfying \eqref{eq: coercive z} and $\lim_{n \to \infty} \A_k(\qn; 2\pi)=0$, since
$$ \A_k(\qn; 2\pi) \ge \int_{0}^{2\pi} U(\qn) \,dt = \int_0^{2\pi} \sum_{\{i< j\} \subset \N} \frac{1}{|\qn_i-\qn_j|} \,dt,$$
the following must hold 
$$ \min_{\{i < j\} \subset \N, \; t \in [0, 2\pi]} |\qn_i(t)-\qn_j(t)| \to \infty, \; \text{ as } n \to \infty,$$
which then implies $\| \qn \|_{\hnm} \to \infty$. Therefore $\qn$ does not converge to any loop in $\ldn_{\xi}$, and this finishes our proof of property (b). 

(c). Like before we will only consider the case $\xi=\xi(\fn_k^+)$. Recall that $\qn(t)=(\zt^n_i(t), z^n_i(t))_{i \in \N}$, first we will show after passing to a subsequence, $z^n_i(t)$ converges uniformly to $0$, for each $i \in \N$. By property (a) in Lemma \ref{lm: no coercive}, $\|z^n_i\|_{\hnm}$ has an finite upper bound independent of $i$ and $n$. Then after passing to a subsequence, for each $i$, $z^n_i(t)$ converges uniformly to a function $z_i(t)$ on $[0, 2\pi]$. We claim $z_i(t)$ is constant for all $t \in [0, 2\pi]$. Otherwise, let's assume there exist $0 \le t_1 < t_2 \le 2\pi$, such that 
$$  |z^n_i(t_1) -z^n_i(t_2)| \ge C >0, \; \; \text{ for } n \text{ large enough}. $$
Then by Cauchy-Schwarz inequality, for $n$ large enough, 
$$ \int_0^{2\pi} |\zd^n_i|^2 \,dt \ge \frac{1}{2\pi} \left( \int_0^{2\pi} |\zd^n_i| \,dt \right)^2 \ge \frac{C^2}{2\pi} > 0. $$
This then implies 
$$\A_k(\qn; 2\pi) \ge \frac{N}{2} \int_0^{2\pi} |\zd^n_i|^2 \,dt \ge \frac{N C^2}{4 \pi} >0,$$
which is absurd. This proves the claim that $z_i(t)$ is a constant, for all $t$. Meanwhile to satisfy condition \eqref{eq: coercive z}, this constant must be zero.

Meanwhile as $|\qn_i(t)| \ge |\zt^n_i(t)|$, for any $i$ and $n$. By property (b) in Lemma \ref{lm: no coercive}, $|\zt^n_i(t)| \to \infty$ uniformly on $\rr$, for each $i \in \N$, then so is $|\qn_i(t)|$. This finishes our proof of property (c). 

With the above result, a similar argument as in the proof of property (a) above can show $\frac{\qn(t)}{|\qn(t)|}$ converges uniformly to a rotating regular $N$-gon, after passing to a subsequence, and since $\qn \in \ldn_{\xi(\fn_k^{+})}$, after normalization the rotating regular $N$-gon must be $e^{-\J kt} \frac{\fn_k^{+}}{|\fn_k^{+}|}$.

\end{proof}

%---------------------------------------------

\section{Connecting planar linear chains in rotating frame} \label{sec: rotate}
In this section we will study the properties of action minimizers of $\ao$ in $\ldn_{\xi}$ (under certain coercive conditions) as $\om$ changes from $0$ to $N$. First we obtain a result analogous to Lemma \ref{lm: monotone fixed}, although it only holds for the $z$-component. 

\begin{lm}
\label{lm: monotone rotate} For any $\om \in \rr$ and $\xi \in \Xi_N$, if $q \in \ldn_\xi$ is a minimizer of $\ao$ in $\ldn_\xi$, then
\begin{equation}
\label{eq: monotone rotate} \text{either } z_0(t_1) \le z_0(t_2) \text{ or }z_0(t_1) \ge z_0(t_2) \text{ always holds, } \forall 0 \le t_1 \le t_2 \le \pi. 
\end{equation}
\end{lm}

\begin{proof}
Let $\qt(t) \in \ldn_\xi$ be defined as in the proof of Lemma \ref{lm: monotone fixed}, the desired result follows from the same argument given there, once we notice that for any $\om \in \rr$, 
\begin{equation}
\label{eq: K om q q tilde} \int_0^{\pi/N} K_{\om}(q, \qd) \,dt = \int_0^{\pi/N} K_{\om}(\qt, \dot{\qt}) \,dt, \;\; \forall i \in \N. 
\end{equation}
\end{proof}

\begin{rem}
The reason that a corresponding result for the $x$-component does not hold when $\om \ne 0$ is \eqref{eq: K om q q tilde} generally does not hold for $\om \ne 0$, if we exchange the role of $x$ and $z$ coordinates during the definition of $\tilde{q}$. 
\end{rem}

\begin{lm}
\label{lm: strict monotone rotate} For any $\xi \in \Xi_N$ and $\om \in \rr$, if $q \in \ldn_\xi$ is a minimizer of $\ao$ among all loops in $\ldn_{\xi}$ satisfying \eqref{eq: coercive z} (and \eqref{eq: coercive x}, if $\om = kN$, $k \in \zz$), then the following results hold. 
\begin{enumerate}
\item[(a).] If $\om=0$ and $x_0(t)$ is not a constant for all $t \in \rr$, then $q$ is a collision-free $2\pi$-periodic solution of \eqref{eq:nbody}. Moreover $\xd_0(t)=0$, if and only if $t \in \{0, \pi\}$ and 
\begin{equation}
  \label{eq: xd 0 > 0} \begin{cases}
\xd_0(t) >0 \; (\text{resp.} <0), & \; \text{ if } t \in (0, \pi), \\
\xd_0(t) <0 \; (\text{resp.} >0), & \; \text{ if } t \in (\pi, 2\pi). 
\end{cases} 
  \end{equation} 
\item[(b).] For any $\om \in \rr$, if $z_0(t)$ is not a constant for all $t \in \rr$, then $q$ is a collision-free $2\pi$-periodic solution of \eqref{eq: nbody rotate}. Moreover $\zd_0(t)=0$, if and only if $t \in \{0, \pi\}$ and 
\begin{equation}
 \label{eq: zd 0 > 0} \begin{cases}
\zd_0(t) >0 \; (\text{resp.} <0), & \; \text{ if } t \in (0, \pi), \\
\zd_0(t) <0 \; (\text{resp.} >0), & \; \text{ if } t \in (\pi, 2\pi). 
\end{cases} 
 \end{equation}
\end{enumerate}
\end{lm}

\begin{proof} We will give a detailed proof of property (b), while property (a) can be proven similarly. 

Let's assume $N=2n$ (the proof for $N=2n+1$ is similarly and will be omitted) . By Lemma \ref{lm: monotone rotate}, $z_0(t)$ satisfies \eqref{eq: monotone rotate}. Without loss of generality, we will assume
\begin{equation}
\label{eq: z0 t1 le z0 t2} z_0(t_1) \le z_0(t_2), \; \; \forall 0 \le t_1 \le t_2 \le \pi.
\end{equation}
Due to the $D_N$-symmetry, this is implies \eqref{eq: symmetric boundary moment 1} and 
\begin{equation}
\label{eq: symm monotone z} \forall 0 \le t_1 < t_2 \le \frac{\pi}{N}, \;\; \begin{cases}
z_i(t_1) \le z_i(t_2), \; & \text{ if } i \in \{0, \dots, n-1 \}, \\
z_i(t_1) \ge z_i(t_2), \; & \text{ if } i \in \{n, \dots, N-1 \}. 
\end{cases}
\end{equation}

The $\xi$-topological constraints are only imposed on the boundary moments of the fundamental domain $[0, \pi/N]$. Hence for any $t \in (0, \pi/N)$, $q(t)$ is a local minimizer of $\ao$ among all paths defined in a small neighborhood of $t$ with the same fixed ends. By the result of Marchal and Chenciner \cite{C02}, $q(t)$ must be collision-free, for any $t \in (0, \pi/N)$. Therefore it satisfies \eqref{eq: nbody rotate}. As a result, $\zd_i(t)$ is well-defined, for any $t \in (0, \pi/N)$ and $i \in \N$. By \eqref{eq: symm monotone z}, it means
\begin{equation}
\label{eq: zd ge 0} \forall t \in (0, \frac{\pi}{N}), \;\; \begin{cases}
\zd_i(t) \ge 0, \; & \text{ if } i \in \{0, \dots, n-1 \}, \\
\zd_i(t) \le 0, \; & \text{ if } i \in \{n, \dots, N-1 \}. 
\end{cases}
\end{equation}

Since $z_0(t), \forall t \in \rr$, is not a constant, by \eqref{eq: z0 t1 le z0 t2}, 
$$  z_n(0) - z_0(0)=z_0(\pi)- z_0(0) >0. $$
As a result, for $\dl>0$ small enough, there is a positive constant $C_1$, such that 
\begin{equation}
\label{eq: zn - z0} z_n(t)-z_0(t) \ge C_1, \; \forall t \in [0, \dl]. 
\end{equation}
Using this, we will show the inequalities in \eqref{eq: zd ge 0} must be strict. Without loss of generality, let's assume $\zd_k(t_0)=0$, for some $t_0 \in (0, \pi/N)$ and $k \in \{0, \dots, n-1\}$ (the cases for $k \in \{n, \dots, N-1\}$ can be proven similarly). For $\ep>0$ small enough, we define a new path $\qe \in \ldn_\xi$ by
$$ \qe_k(t) = \begin{cases}
q_k(t) - \ep^2 \e_3, \; & \forall t \in [0, t_0-\ep], \\
q_k(t)+ (t-t_0)(2\ep-|t-t_0|)\e_3, \; & \forall t \in [t_0-\ep, t_0 +\ep], \\
q_k(t) + \ep^2 \e_3, \; & \forall t \in [t_0+ \ep, \frac{\pi}{N}], 
\end{cases} $$
$$ \qe_i(t)= \begin{cases}
q_i(t)- \ep^2 \e_3, & \text{ if } i \in \{0, \dots, k-1 \} \cup \{N-k, \dots, N-1\}, \\
q_i(t) + \ep^2 \e_3, & \text{ if } i \in \{k+1, \dots, N-1-k\}, 
\end{cases} \; \forall t \in [0, \frac{\pi}{N}]. $$
We may also need to shift $\qe(t)$ by a constant along the $z$-axis to make sure it satisfies \eqref{eq: coercive z}. Since $\zd_k(t_0)=0$, there is a constant $C_2>0$ independent of $\ep$, such that $ |\zd_k(t)| \le C_2 |t-t_0|$, for $|t- t_0|$ small enough. Then by a simple computation, 
\begin{equation} \label{eq: K zd 0} \begin{split}
\int_{0}^{\frac{\pi}{N}} K_{\om}(\qe, \qd^{\ep}) & - K_{\om}(q, \qd)\,dt  = \ey \int_{t_0-\ep}^{t_0+\ep} |\dot{z}_k^{\ep}|^2- |\dot{z}_k|^2 \,dt \\
& = 2 \int_{t_0-\ep}^{t_0+\ep} (\ep-|t-t_0|)^2 + \zd_k(t)(\ep-|t-t_0|) \,dt \le C_3 \ep^3,
\end{split}
\end{equation}
where $C_3>0$ is a constant independent of $\ep$. This controls the change in kinetic energy. For potential energy, we notice that by \eqref{eq: symmetric boundary moment 1} and \eqref{eq: symm monotone z},
\begin{equation}
\label{eq: qei-qej ge qi-qj} |\qe_i(t)-\qe_j(t)| \ge |q_i(t)- q_j(t)|, \; \forall t \in [0, \frac{\pi}{2}], \; \forall \{i \ne j\} \subset \N. 
\end{equation}
Moreover we can always find a $\dl >0$ small enough (in particular, when $k=0$, we need $\dl < t_0-\ep$), such that for any $t \in [0, \dl]$,
\begin{equation}
\label{eq: qen - qe0}
\begin{split}
|\qe_n(t)-\qe_0(t)|^2 & = (x_n(t)-x_0(t))^2+(y_n(t)-y_0(t))^2 + (z_n(t)-z_0(t)+2\ep^2)^2 \\
					 & = |q_n(t)-q_0(t)|^2+4(z_n(t)-z_0(t))\ep^2 + 4\ep^4. \end{split}
\end{equation}
Meanwhile there is constant $C_4>0$ independent of $\ep$, such that 
\begin{equation}
\label{eq: qn-q0} |q_n(t)- q_0(t)|^{-1} \ge C_4, \; \forall t \in [0, \dl].
\end{equation}

Combining this with \eqref{eq: zn - z0} and \eqref{eq: qen - qe0}, for any $t \in [0, \dl]$, we have 
\begin{equation} \label{eq: differe qn-q0}
\begin{split}
& \frac{1}{|\qe_n(t) -\qe_0(t)|} -\frac{1}{|q_n(t)-q_0(t)|} \\
 & = \frac{1}{|q_n(t)-q_0(t)|} \left[ \left(1+\frac{4(z_n(t)-z_0(t)) \ep^2}{|q_n(t)-q_0(t)|^2} + \frac{4\ep^4}{|q_n(t)-q_0(t)|^2}\right)^{-\ey} -1 \right] \\
& \le -C_5 \ep^2,
\end{split}
\end{equation}
where $C_5>0$ is independent of $\ep$. By \eqref{eq: qei-qej ge qi-qj} and \eqref{eq: differe qn-q0}, 
\begin{equation}
\begin{split}
\int_0^{\frac{\pi}{N}} U(\qe)- U(q) \,dt & \le \int_{0}^{\dl} \frac{1}{|\qe_n(t) -\qe_0(t)|} -\frac{1}{|q_n(t)-q_0(t)|} \,dt \\
& \le \int_0^{\dl} -C_5 \ep^2 \,dt = -C_5 \dl \ep^2. 
\end{split}
\end{equation}
As a result, for $\ep$ small enough, 
$$ \ao(\qe; \pi/N) - \ao(q; \pi/N) \le C_3 \ep^3 -C_5 \dl \ep^2 < 0,$$
which is a contradiction to the minimization property of $q$. This shows all the inequality in \eqref{eq: zd ge 0} must be strict, i.e.
\begin{equation}
\label{eq: zd > 0} \forall t \in (0, \frac{\pi}{N}), \;\; \begin{cases}
\zd_i(t) > 0, \; & \text{ if } i \in \{0, \dots, n-1 \}, \\
\zd_i(t) < 0, \; & \text{ if } i \in \{n, \dots, N-1 \}. 
\end{cases}
\end{equation}

For a moment, let's assume $q(t)$ is collision-free at both $t=0 $ and $\pi/N$. Then $\zd_i(t)$ are well-defined, for any $t \in \{0, \pi/N\}$ and $i \in \N$, and a similar argument as above will show 
\begin{equation} \label{eq: zd > 0 boudnary}
\forall t \in \{0, \frac{\pi}{N}\}, \;\; \begin{cases}
\zd_i(t) > 0, \; & \text{ if } i \in \{1, \dots, n-1 \}, \\
\zd_i(t) < 0, \; & \text{ if } i \in \{n+1, \dots, N-1 \}. 
\end{cases}
\end{equation}
Due to the $D_N$-symmetry, \eqref{eq: zd > 0} and \eqref{eq: zd > 0 boudnary} immediately imply \eqref{eq: zd 0 > 0}. 

Meanwhile by action minimization property of $q$, $q_0(t)$ and $q_n(t)$ must hit the $xz$-plane perpendicularly, which means $\zd_0(0) = \zd_n(0) =0.$ Again by the $D_N$-symmetry, this implies $\zd_0(0)= \zd_0(\pi)=0$. 

After the above argument, the only thing left for us is to show $q(t)$ is collision-free at the boundary moments $t=0$ and $\pi/N$. By a contradiction argument, let's assume $q(0)$ is not collision-free (the proof for $q(\pi/N)$ is similar and will be left to the readers). By \eqref{eq: zd > 0}, only binary collisions are possible at $t=0$. Assuming there is a binary collision between $m_j$ and $m_k$ at $t=0$, for some $j < k$ (due to the $D_N$-symmetry, $j+k=N$). Notice that we may have more than one isolated binary collision at this moment. 

Let's us follow the notations set up in Section \ref{sec: lemmas}. By Proposition \ref{prop:angle}, $m_i$, $i \in \{j, k\}$, approaches to the binary collision along a definite direction given by the unit vector $(1, \phi^+_i, \tht^+_i)$ under spherical coordinates. Depending on the value of $\tht_i^+$, either Lemma \ref{lm:dfm1} or \ref{lm:dfm4} will be used to get a contradiction. To ensure the $\xi$-topological constraints will be satisfied in our argument, we need to know the precise value of  $\xi_{2j}$ (the $2j$-th component of $\xi$), as it determines the relative position of $m_j$ and $m_k$ along the $y$-axis. Without loss of generality, let's assume $\xi_{2j}=1$. This means for any $q^* \in \ldn_{\xi}$, $y^*_j(0) \ge 0 \ge y^*_k(0).$

First, if $\tht_j^+ \ne -\pi/2$, then by Lemma \ref{lm:dfm1}, for $\ep>0$ small enough, we can make a local deformation of $q$ to get a new path $\qe \in H^1([0, \pi/N], \rr^{3N})$ satisfying $\ao(\qe; \pi/N) < \ao(\qe; \pi/N)$ and $y^{\ep}_j(0)= -y^{\ep}_k(0)>0$. In particular, $\qe \in \ldn_{\xi}$, which is absurd. Notice that here and in the following we may need to shift $\qe$ by a proper constant along the $z$-axis to make sure \eqref{eq: coercive z} are satisfied. Similarly we may also need to shift the deformed path along the $x$-axis by a proper constant, when \eqref{eq: coercive x} is required to be satisfied. 

Second, if $\tht_j^+= -\pi/2$, by \eqref{eq: zd > 0}, $q(t)$, $t \in [0, \pi/N]$ is $z$-separated (by $m_j$ and $m_k$) with $z_j(\pi/N) > z_k(\pi/N)$. Then for $\ep>0$ small enough, by Lemma \ref{lm:dfm2}, we can find a new path $\qe \in \ldn_{\xi}$ with $\ao(\qe; \pi/N) < \ao(q: \pi/N)$, which is a contradiction.  
\end{proof}

\begin{lm}
\label{lm: collision} For any $\xi \in \Xi_N$ and $\om$, let $q \in \ldn_\xi$ be a minimizer of $\ao$ among all loops in $\ldn_{\xi}$ satisfying \eqref{eq: coercive z} (and \eqref{eq: coercive x}, if $\om = kN$, $k \in \zz$), if $\Delta^{-1}(q)$ is not empty, then the following must hold. 
\begin{enumerate}
\item[(a).] $z_i(t) \equiv 0$, $\forall t \in \rr$ and $ \forall i \in \N$.
\item[(b).] $\Delta^{-1}(q) \subset \{ t = \ell \pi/N: \; \ell \in \zz\} $, and $q(t)$ is collision-free and satisfies \eqref{eq: nbody rotate}, for any $t \in \rr \setminus \Delta^{-1}(q)$.
\item[(c).] For any $t \in \Delta^{-1}(q)$, if $q(t)$ has an $\I$-cluster collision for some $\I \subset \N$, then $|\I|=2$. Moreover when $t=0$ or $\pi/N$ and $\I = \{j, k\}$, the it must satisfies \eqref{eq; possible binary collision}.
\end{enumerate}
\end{lm}

\begin{proof}
We will only give a detailed proof for $N=2n$, while the proof for $N=2n+1$ is similar and will be omitted. 

(a). Assume the result of property (a) does not hold, then $z_0(t) \ne \text{Constant}$, for all $t \in \rr$. By Lemma \ref{lm: strict monotone rotate}, $q$ must be collision-free, which is a contradiction. 

(b). Like the argument given in Lemma \ref{lm: strict monotone rotate}, as the $\xi$-topological constraints are essentially imposed on the boundary moments of the fundamental domain $[0, \pi/N]$, $q(t)$ is collision-free and satisfies equation \eqref{eq: nbody rotate}, for any $t \in (0, \pi/N)$. By the $D_N$-symmetry, the same result must hold for any $t \in \rr \setminus \{ t = \ell \pi/N: \; \ell \in \zz\}$ as well. 

(c). First let us assume $t=0 \in \Delta^{-1}(q)$ and  $q(0)$ has an $\I$-cluster collision. Notice that $m_0$ does not collide with any other mass, when $t=0$. Otherwise we may choose a $\tau =(\tau_i)_{i \in \N} \in \mf{T}$ with $\tau_0=-1$ and $\tau_i=0$, $\forall i \ne 0$. Then for $\ep>0$ small enough, using Lemma \ref{lm:dfm4}, we can get a new path $\qe \in \ldn_{\xi}$ with $\ao(\qe;\pi/N) < \ao(q; \pi/N)$, which is absurd. By a similar argument, one can show $m_n$ does not collide with any other mass either, when $t=0$. 

Now we will show $|\I|=2$. By a contradiction argument, let us say $| \I | \ge 3$. First let us consider the case that there is an $i \in \{1, \dots, n-1\}$, such that  $\{i, N-i\} \subset \I$. As $|\I| \ge 3$, there is a $j \in \I \setminus \{i, N-i\}$. By \eqref{eq: symmetric boundary moment 1}, $q_i(0)=\R_{xz}q_{N-i}(0)$, this implies the $\I$-cluster collision must occur in the $xz$-plane. Since $j \ne 0$ or $n$, \eqref{eq: symmetric boundary moment 1} implies $q_j(0)= q_{N-j}(0)$. As a result, $\{i, j, N-i, N-j\} \subset \I.$ We will choose a $\tau=(\tau_k)_{k \in \N} \in \mf{T}$ with each $\tau_k=0$ except the following:  
$$ \tau_i= \tau_{N-i}=1, \; \tau_{j}=\tau_{N-j}=-1. $$
Then for $\ep>0$ small enough, by Lemma \ref{lm:dfm4}, we can find a new path $\qe \in \ldn_{\xi}$ with $\ao(\qe; \pi/N) < \ao(q; \pi/N)$, which is a contradiction. 

Now let's consider the case that $\{i, N-i \} \not\subset \I$, for any $i \in \{1, \dots, n-1\}$. Then $i+j \ne N$, for any $\{i \ne j \} \subset \I$. As a result, $q_i(0)=q_j(0)$ and \eqref{eq: symmetric boundary moment 1} implies $q_{N-i}(0)= q_{N-j}(0)$. This means there must be a $\I'$-cluster collision at the moment $t=0$ with $\{N-i, N-j\} \subset \I'$ and $\I \cap \I'=\emptyset$. Then a contradiction can be reach by Lemma \ref{lm:dfm4} with the same $\tau$ we just used. This proves our claim that $|\I|=2$. Notice that the same argument we just gave actually also implies $\I = \{i, N-i\}$ for some $i \in \{1, \dots, n-1\}$. This finishes our proof for $t=0$ being a collision moment. 

The proof is similar, when $t=\pi/N$ is a collision moment. For any other collision moment $ t \in \Delta^{-1}(q) \setminus \{0, \pi/N\}$, the result follows directly from the definition of $D_N$-symmetry.  
\end{proof}
%The only thing left for us is to show that if such a binary collision exists between $m_j$ and $m_{N-j}$ at the moment $t=0$, then it is regularizable in the sense of Levi-Civita. To show this, we need to know the value of $\xi_{2j}$, without loss of generality, let's assume $\xi_{2j}=1$. Again following the notations set up in Section \ref{sec: lemmas}, we must have $\tht^+_j(0)=-\frac{\pi}{2} (\text{mod } 2\pi)$ and $\tht^+_{N-j}= \frac{\pi}{2} (\text{mod } 2\pi)$. Otherwise a contradiction can be reached by Lemma \ref{lm:dfm1}, as we argued in the proof of Lemma \ref{lm: strict monotone rotate}. 

Using the above results, we can prove Theorem \ref{thm: linear chain rotate}, \ref{thm: regularization} and \ref{thm: lc rotate extra sym}.
\begin{proof}[Proof of Theorem \ref{thm: linear chain rotate}] (a). When $\om=0$, the result follows from Theorem \ref{thm: linear chain}. When $\om=N$, it follows from Theorem \ref{thm: linear chain} as well. This is because with frequency $N$, from the moment $t=0$ to $t=\frac{\pi}{N}$, the $x$ and $y$-axis rotate around the $z$-axis by $\pi$, so they come back to the original line but with reversed directions. As a result, $q(t) \in \ldn_{\xi}$ in the rotating frame with frequency $\om=N$, if and only if $\ej q(t) \in \ldn_{\xi^*}$, with $\xi^*$ defined as in \eqref{eq: xi star}. 

(b) \& (c). These two properties follows directly from Lemma \ref{lm: strict monotone rotate} and \ref{lm: collision}.  
\end{proof}

%Although the statement of Theorem \ref{thm: sbc reg} does not imply the desired result, the proof of it given in \cite{EB96} does, which we will explain in the following. 

\begin{proof}[Proof of Theorem \ref{thm: regularization}] Since $\qo(t)$ is an action minimizer of $\ao$ with collision, by Theorem \ref{thm: linear chain rotate} in the original non-rotating frame $q(t)= \ej \qo(t)$ is a collision solution of \eqref{eq:nbody} containing only binary collisions and the set of collision moments 
$$ \Delta^{-1}(q) \subset \{ t= \ell \pi: \; \ell \in \zz \}. $$
Without loss of generality, let's assume $t=0$ is a collision moment with $M$ ($1 \le M \le [N/2]$) pairs of binary collision: 
$$ q_{i_j}(0) = q_{N-i_j}(0), \;\; i_j \in \N \text{ for } j=1, \dots, M. $$ 
Recall that by Theorem \ref{thm: linear chain rotate}, a binary collision can only happen between two masses with their indices satisfying \eqref{eq; possible binary collision}. %Set $\I = \{i_j, N-i_j| \; j=1, \dots, M \}$. 

For each $1 \le j \le M$, following the notations from Section \ref{sec: lemmas}, we set 
$$ q_{c_j}(t):= \frac{1}{2}(q_{i_j}(t)+q_{N-i_j}(t)); \;\; \qf_{j}(t)=(\xf_{j}, \yf_{j}, \zf_{j})(t):=q_{i_j}(t)-q_{c_j}(t),$$
and in the spherical coordinates $(r, \phi, \tht)$ with $r \ge 0, \phi \in [0, \pi]$ and $\tht \in \rr$, we have
$$ \xf_{j}=r_{j} \sin \phi_j \cos \tht_j, \; \yf_j = r_j \sin \phi_j \sin \tht_j, \zf_j= r_j \cos \phi_j. $$
Moreover we define the energy of the sub-system consisting of $m_{i_j}$ and $m_{N-i_j}$ as 
\begin{equation*}
\label{eq: energy sub system} E_j(t)=E_j(q(t)):= \ey (|\qd_{i_j}(t)|^2 + |\qd_{N-i_j}(t)|^2)- \frac{1}{|q_{i_j}(t)-q_{N-i_j}(t)|}
\end{equation*}

We can always find a $\dl>0$ small enough, such that $q_i(t) \in C^2((-2\dl, 2\dl), \rr^3)$, $\forall i \notin \cup_{j=1}^M \{i_j, N-i_j \},$ as they represent the motions of masses not involved in any collision. Meanwhile for each $j=1, \dots, M$,
\begin{align*}
r_j(t) \in & C^0((-2\dl, 2\dl), \rr) \cap C^2((-2\dl, 0), \rr) \cap C^2((0, 2\dl), \rr);  \\
\phi_j(t) \in & C^2((-2\dl, 0), \rr) \cup C^2((0, 2\dl), \rr);  \\
\tht_j(t) \in & C^2((-2\dl, 0), \rr) \cup C^2((0, 2\dl), \rr); .
\end{align*}
Despite of the binary collision singularities, for each pair of $\{m_{i_j}, m_{N-i_j}\}$, their center of mass still satisfies
\begin{equation}
\label{eq: center of binary collision C1} q_{c_j}(t) \in C^2((-2\dl, 2\dl), \rr^3), 
\end{equation}
for a proof see \cite[Remark 4.10]{FT04}.

Following \cite{EB96}, using McGehee transformation \cite{MG74}, one can blow up the simultaneous binary collisions to certain manifold, which will be called the collision manifold. The collision manifold becomes boundaries of the phase space (after McGehee transformation). Then one can extend the vector field \eqref{eq Ham vector field} to the collision manifold, which are invariant under the extended flow. The extended flow on the collision manifold can be understood completely: the problem becomes $M$ pairs of decoupled two body problems on the collision manifold, and each pair of masses involved in the binary collisions make a complete revolution on a fixed plane around its center of mass and the energy of the sub-system is a first integral. With this one can find the unique ejection orbit associated with a given collision orbit. The proof is quite long and technical, in our setting we summerize it as following: 

$(q,\qd)(t)$, $t \in (0, 2\dl)$ is the unique ejection orbit associated with $(q, \qd)(t)$, $t \in (-2\dl, 0)$, if the following conditions hold for each $1 \le j \le M$, 
\begin{align}
\label{eq: phi reg} \lim_{t \to 0^-} \phi_j(t) &= \lim_{t \to 0^+} \phi_j(t), \;\;  \lim_{t \to 0^-} \dot{\phi}_j(t) = \lim_{t \to 0^+} \dot{\phi}_j(t)=0; \\
\label{eq: tht reg} \lim_{t \to 0^-} \tht_j(t) &=\lim_{t \to 0^+} \tht_j(t) (\text{mod} 2\pi), \;\;  \lim_{t \to 0^-} \dot{\tht}_j(t) =  \lim_{t \to 0^+} \dot{\tht}_j(t)=0; \\
\label{eq: energy reg} \lim_{t \to 0^-}E_j(t) &= \lim_{t \to 0^+} E_j(t). 
\end{align}

We explain why these conditions hold for the minimizer $q$: first, by property (c) in Theorem \ref{thm: linear chain rotate}, $q_i(t)$ belongs to the $xy$-plane, for any $i$ and $t$. Hence $\phi_j(t)=\pi/2$, for any $t$ an $1 \le j \le M$, which immediately implies \eqref{eq: phi reg}; second, the first equation in \eqref{eq: tht reg} following from Lemma \ref{lm:dfm1}, as otherwise using this lemma we can make a small local deformation of the collision solution near the collision moment and get a new path with strict smaller action value, and the second equation in \eqref{eq: tht reg} follows from property (b) in Proposition \ref{prop:angle} (the proposition is stated for a ejection solution, but the same holds for a collision solution as well); third, since $q$ is an action minimizer, by results from \cite[Section 4]{FT04}, we have $E_j(t) \in C^0((-2\dl, 2\dl), \rr)$, for each $1 \le j \le M$, which clearly implies \eqref{eq: energy reg}. 

We have proved $(q, \qd)(t)$, $t \in (-2\dl, 2\dl)$ is $C^0$ block-regularizable. The same argument can be applied to any other collision moment and this finishes our proof.
\end{proof}

%The result follows from a similar argument given in the proof of Theorem \ref{thm: linear chain extra sym}. Of course here instead of Lemma \ref{lm: monotone fixed} and \ref{lm: strict monotone fixed}, we will use Lemma \ref{lm: monotone rotate} and \ref{lm: strict monotone rotate}. Like before, the same problem that the path $\qt(t)$ constructed as in the proof of Lemma \ref{lm: monotone fixed} may not belong to $\Lmd^{H_N}_{\xi}$ will occur when we try to use Lemma \ref{lm: monotone rotate}. Again to solve this problem, we need to shift the path $\qt(t)$ by a proper constant along the $z$-axis. Notice that this does not change the action of $\ao$. 

\begin{proof}[Proof of Theorem \ref{thm: lc rotate extra sym}] The fact that a minimizer satisfies all the properties in Theorem \ref{thm: linear chain rotate} and \ref{thm: regularization} follows from the same arguments as before and will not be repeated. 
\end{proof}

%------------------------------

%\section{Non-Newtonian potentials} \label{sec: Non-Newtonian}

%---------------------------------------

\emph{Acknowledgements.} The author thanks Alain Chenciner, Jacques F\'ejoz for valuable discussions and permission to use the numerical pictures from their paper \cite{CF09}. He thanks Richard Montgomery and Carles Sim\'o for their interests and comments on this work. The main part of the work was done when the author was a postdoc at Ceremade, University of Paris-Dauphine and IMCCE, Paris Observatory. He thinks the hospitality of both institutes and financial support of FSMP. Part of the work was done, when the author was a visitor at Shandong University. He thanks Xijun Hu for financial support through NSFC(No.11425105).

\bibliographystyle{abbrv}
\bibliography{RefLinearChain}

\def\cprime{$'$}
\begin{thebibliography}{10}

\bibitem{BFT08}
V.~Barutello, D.~L. Ferrario, and S.~Terracini.
\newblock {Symmetry groups of the planar three-body problem and
  action-minimizing trajectories}.
\newblock {\em Arch. Ration. Mech. Anal.}, 190(2):189--226, 2008.

\bibitem{BT04}
V.~Barutello and S.~Terracini.
\newblock {Action minimizing orbits in the {$n$}-body problem with simple
  choreography constraint}.
\newblock {\em Nonlinearity}, 17(6):2015--2039, 2004.

\bibitem{Ch07}
K.-C. Chen.
\newblock Removing collision singularities from action minimizers for the
  {$N$}-body problem with free boundaries.
\newblock {\em Arch. Ration. Mech. Anal.}, 181(2):311--331, 2006.

\bibitem{C02}
A.~Chenciner.
\newblock {Action minimizing solutions of the {N}ewtonian {$n$}-body problem:
  from homology to symmetry}.
\newblock In {\em {Proceedings of the {I}nternational {C}ongress of
  {M}athematicians, {V}ol. {III} ({B}eijing, 2002)}}, pages 279--294. Higher
  Ed. Press, Beijing, 2002.

\bibitem{C04}
A.~Chenciner.
\newblock Are there perverse choreographies?
\newblock In {\em New advances in celestial mechanics and {H}amiltonian
  systems}, pages 63--76. Kluwer/Plenum, New York, 2004.

\bibitem{C08}
A.~Chenciner.
\newblock Four lectures on the {$N$}-body problem.
\newblock In {\em Hamiltonian dynamical systems and applications}, NATO Sci.
  Peace Secur. Ser. B Phys. Biophys., pages 21--52. Springer, Dordrecht, 2008.

\bibitem{CF08}
A.~Chenciner and J.~F\'ejoz.
\newblock The flow of the equal-mass spatial 3-body problem in the neighborhood
  of the equilateral relative equilibrium.
\newblock {\em Discrete Contin. Dyn. Syst. Ser. B}, 10(2-3):421--438, 2008.

\bibitem{CF09}
A.~Chenciner and J.~F{\'e}joz.
\newblock Unchained polygons and the {$N$}-body problem.
\newblock {\em Regul. Chaotic Dyn.}, 14(1):64--115, 2009.

\bibitem{CFM05}
A.~Chenciner, J.~F{\'e}joz, and R.~Montgomery.
\newblock Rotating eights. {I}. {T}he three {$\Gamma_i$} families.
\newblock {\em Nonlinearity}, 18(3):1407--1424, 2005.

\bibitem{CGMS02}
A.~Chenciner, J.~Gerver, R.~Montgomery, and C.~Sim{\'o}.
\newblock {Simple choreographic motions of {$N$} bodies: a preliminary study}.
\newblock In {\em {Geometry, mechanics, and dynamics}}, pages 287--308.
  Springer, New York, 2002.

\bibitem{CM00}
A.~Chenciner and R.~Montgomery.
\newblock {A remarkable periodic solution of the three-body problem in the case
  of equal masses}.
\newblock {\em Ann. of Math. (2)}, 152(3):881--901, 2000.

\bibitem{EB96}
M.~S. ElBialy.
\newblock The flow of the {$N$}-body problem near a
  simultaneous-binary-collision singularity and integrals of motion on the
  collision manifold.
\newblock {\em Arch. Rational Mech. Anal.}, 134(4):303--340, 1996.

\bibitem{Fe06}
D.~L. Ferrario.
\newblock Symmetry groups and non-planar collisionless action-minimizing
  solutions of the three-body problem in three-dimensional space.
\newblock {\em Arch. Ration. Mech. Anal.}, 179(3):389--412, 2006.

\bibitem{FT04}
D.~L. Ferrario and S.~Terracini.
\newblock {On the existence of collisionless equivariant minimizers for the
  classical {$n$}-body problem}.
\newblock {\em Invent. Math.}, 155(2):305--362, 2004.

\bibitem{FGN11}
G.~Fusco, G.~F. Gronchi, and P.~Negrini.
\newblock {Platonic polyhedra, topological constraints and periodic solutions
  of the classical {$N$}-body problem}.
\newblock {\em Invent. Math.}, 185(2):283--332, 2011.

\bibitem{Go77}
W.~B. Gordon.
\newblock {A minimizing property of {K}eplerian orbits}.
\newblock {\em Amer. J. Math.}, 99(5):961--971, 1977.

\bibitem{KS65}
P.~Kustaanheimo and E.~Stiefel.
\newblock Perturbation theory of {K}epler motion based on spinor
  regularization.
\newblock {\em J. Reine Angew. Math.}, 218:204--219, 1965.

\bibitem{Mc00}
C.~Marchal.
\newblock The family {$P_{12}$} of the three-body problem---the simplest family
  of periodic orbits, with twelve symmetries per period.
\newblock {\em Celestial Mech. Dynam. Astronom.}, 78(1-4):279--298 (2001),
  2000.
\newblock New developments in the dynamics of planetary systems (Badhofgastein,
  2000).

\bibitem{MS00}
R.~Mart\'\i~nez and C.~Sim\'o.
\newblock The degree of differentiability of the regularization of simultaneous
  binary collisions in some {$N$}-body problems.
\newblock {\em Nonlinearity}, 13(6):2107--2130, 2000.

\bibitem{MVV13}
E.~Mateus, A.~Venturelli, and C.~Vidal.
\newblock Quasiperiodic collision solutions in the spatial isosceles three-body
  problem with rotating axis of symmetry.
\newblock {\em Arch. Ration. Mech. Anal.}, 210(1):165--176, 2013.

\bibitem{MG74}
R.~McGehee.
\newblock Triple collision in the collinear three-body problem.
\newblock {\em Invent. Math.}, 27:191--227, 1974.

\bibitem{Mo99}
R.~Montgomery.
\newblock Figure 8s with three bodies.
\newblock http://people.ucsc.edu/~rmont/Nbdy.html.

\bibitem{Mo02}
R.~Montgomery.
\newblock Action spectrum and collisions in the planar three-body problem.
\newblock In {\em Celestial mechanics ({E}vanston, {IL}, 1999)}, volume 292 of
  {\em Contemp. Math.}, pages 173--184. Amer. Math. Soc., Providence, RI, 2002.

\bibitem{Pa79}
R.~S. Palais.
\newblock {The principle of symmetric criticality}.
\newblock {\em Comm. Math. Phys.}, 69(1):19--30, 1979.

\bibitem{Ca17}
E.~D. R.~Calleja and C.~García-Azpeitia.
\newblock {Symmetries and choreographies in families that bifurcate from the
  polygonal relative equilibrium of the n-body problem}.
\newblock Arxiv 1702.03990, 2017.

\bibitem{Sh11}
M.~Shibayama.
\newblock Minimizing periodic orbits with regularizable collisions in the
  {$n$}-body problem.
\newblock {\em Arch. Ration. Mech. Anal.}, 199(3):821--841, 2011.

\bibitem{Sh14}
M.~Shibayama.
\newblock {Variational proof of the existence of the super-eight orbit in the
  four-body problem}.
\newblock {\em Arch. Ration. Mech. Anal.}, 214(1):77--98, 2014.

\bibitem{Si00}
C.~Sim{\'o}.
\newblock {New families of solutions in {$N$}-body problems}.
\newblock In {\em {European {C}ongress of {M}athematics, {V}ol. {I}
  ({B}arcelona, 2000)}}, volume 201 of {\em {Progr. Math.}}, pages 101--115.
  Birkh{\"a}user, Basel, 2001.

\bibitem{SL92}
C.~Sim\'o and E.~A. Lacomba.
\newblock Regularization of simultaneous binary collisions in the {$n$}-body
  problem.
\newblock {\em J. Differential Equations}, 98(2):241--259, 1992.

\bibitem{Ve01}
A.~Venturelli.
\newblock {Une caract{\'e}risation variationnelle des solutions de {L}agrange
  du probl{\`e}me plan des trois corps}.
\newblock {\em C. R. Acad. Sci. Paris S{\'e}r. I Math.}, 332(7):641--644, 2001.

\bibitem{Ve02}
A.~Venturelli.
\newblock {\em {"Application de la Minimisation de L'action au Probl{\`e}me des
  N Corps dans le plan et dans L'espace,"}}.
\newblock PhD thesis, Universit{\'e} Denis Diderot in Paris, 2002.

\bibitem{Y16s}
G.~Yu.
\newblock {Spatial double choreographies of the Newtonian $2n$-body problem}.
\newblock to apprear in \emph{Archive for Rational Mechanics and Analysis}, on
  arXiv:1608.07956, 2016.

\bibitem{Y15b}
G.~Yu.
\newblock Shape space figure-8 solution of three body problem with two equal
  masses.
\newblock {\em Nonlinearity}, 30(6):2279--2307, 2017.

\bibitem{Y15c}
G.~Yu.
\newblock Simple choreographies of the planar {N}ewtonian {$N$}-body problem.
\newblock {\em Arch. Ration. Mech. Anal.}, 225(2):901--935, 2017.

\end{thebibliography}

\end{document}